\documentclass[12pt, reqno, a4paper]{amsart}
\usepackage[margin=1.2in]{geometry}
\numberwithin{equation}{section}
\usepackage{amssymb,amsfonts,amsthm}
\usepackage[utf8]{inputenc}
\usepackage{listings}
\usepackage{bm}
\usepackage[hyperpageref]{backref}
\usepackage{esint}
\usepackage{color}
\usepackage{bigints}
\usepackage[bookmarks]{hyperref}
\usepackage{siunitx}
\usepackage{hyperref}
\allowdisplaybreaks[4]

\newtheoremstyle{myremark}{10pt}{10pt}{}{}{\bfseries}{.}{.5em}{}

 \newtheorem{thm}{Theorem}[section]
 \newtheorem{cor}[thm]{Corollary}
 \newtheorem{lemma}[thm]{Lemma}
 
 \theoremstyle{definition}
 \newtheorem{defn}{Definition}
 
 \newtheorem{rem}{Remark}
\newtheorem{theorem}{Theorem}
\usepackage{hyperref}

\allowdisplaybreaks[4]

\begin{document}

\title{Boundary Hardy inequality on functions of bounded variation}

\author{ADIMURTHI, PROSENJIT ROY, AND VIVEK SAHU}

\address{ Department of Mathematics and Statistics,
Indian Institute of Technology Kanpur, Kanpur - 208016, Uttar Pradesh, India}

\email{Adimurthi: adiadimurthi@gmail.com, adimurthi@iitk.ac.in}
\email{Prosenjit Roy: prosenjit@iitk.ac.in}
\email{Vivek Sahu: viveksahu20@iitk.ac.in, viiveksahu@gmail.com}

\subjclass[2020]{ 46E35, 26D15, 39B62}

\keywords{Hardy inequality, functions of bounded variation, fractional sobolev space.}

\date{}

\dedicatory{}

\begin{abstract} Classical boundary Hardy inequality, that goes back to 1988, states that if ~$1 < p < \infty,  \ ~\Omega$ is bounded Lipschitz domain, then for all ~$u \in C^{\infty}_{c}(\Omega)$, 
$$\int_{\Omega} \frac{|u(x)|^{p}}{\delta^{p}_{\Omega}(x)}  dx \leq C\int_{\Omega} |\nabla u(x) |^{p}dx,$$ where ~$\delta_\Omega(x)$ is the distance function from ~$\Omega^c$. In this article, we address the long standing open question on the case $p=1$ by establishing appropriate boundary Hardy inequalities in the space of functions of bounded variation.  We first establish appropriate inequalities on fractional Sobolev spaces $W^{s,1}(\Omega)$ and then Brezis, Bourgain and  Mironescu's result on limiting behavior of fractional Sobolev spaces as ~$s\rightarrow 1^{-}$ plays an important role in the proof.  Moreover, we also derive an infinite series Hardy inequality for the case ~$p=1$.
\end{abstract}

\maketitle
\section{Introduction}The classical Hardy inequality for the local case is given by
\begin{equation*}
    \int_{\mathbb{R}^d} \frac{|u(x)|^{p}}{|x|^{p}} dx \leq   \left| \frac{p}{p-d} \right|^{p} \int_{\mathbb{R}^d} |\nabla u(x)|^{p} dx,
\end{equation*}
holds for all ~$u \in C^{\infty}_{c}(\mathbb{R}^d)$ if ~$1<p<d$ and for all ~$u \in C^{\infty}_{c}(\mathbb{R}^d \backslash \{ 0 \})$ if ~$p>d$. Let ~$\Omega$ be a bounded domain in ~$\mathbb{R}^d$, ~ $d \geq 2$, with ~$0 \in \Omega$, we have
\begin{equation}\label{hardy ineq}
    \int_{\Omega} \frac{|u(x)|^{p}}{|x|^{p}} dx \leq \left( \frac{p}{d-p} \right)^{p} \int_{\Omega} |\nabla u(x)|^{p} dx,
\end{equation}
holds for all ~$ u \in C^{\infty}_{c}(\Omega)$ if ~$1 < p < d$ and the constant ~$\left( \frac{p}{d-p} \right)^{p}$ is sharp but never achieved. The inequality analogous to ~\eqref{hardy ineq} for the case ~$p=d=2$ was explored by Leray in ~\cite{leray1933}, and it has been extended to ~$p=d \geq 2$ by ~\cite{adi2002, tertikas2003, tertikas2004}. It can be formulated as follows: Let ~$\Omega \subset \mathbb{R}^d$, where ~$d \geq 2$, be a bounded domain. Then, there exists a constant ~$C=C(d, \Omega, R)$ such that for any ~$u \in W^{1,d}_{0}(\Omega)$,
\begin{equation*}
    \int_{\Omega} |\nabla u(x)|^{d} dx \geq C \int_{\Omega} \frac{|u(x)|^{d}}{|x|^{d}} \left( \ln \frac{R}{|x|} \right)^{- d} dx.
\end{equation*}
where ~$R \geq \sup_{\Omega} (|x|e^{\frac{2}{p}})$.

\smallskip 

J. L. Lewis in ~\cite{lewis1988} proved the boundary Hardy inequality for the local case and established the following result: let ~$\Omega \subset \mathbb{R}^d$ be a bounded domain with Lipschitz boundary and ~$1<p< \infty$, then there exists a constant ~$C=C(d,p, \Omega)>0$ such that
\begin{equation}\label{boundary hardy inequality}
      \int_{\Omega} \frac{|u(x)|^{p}}{\delta^{p}_{\Omega}(x)}  dx \leq C \int_{\Omega} |\nabla u(x) |^{p}dx, \hspace{.3cm} \text{for all} \ u \in C^{\infty}_{c}(\Omega),
\end{equation}
where ~$\delta_{\Omega}$ is the distance function from the boundary of ~$\Omega$  defined by
\begin{equation*}
    \delta_{\Omega}(x) :=  \underset{y \in \partial \Omega}{\min} |x-y|.
\end{equation*} 
Several generalisations and extensions  of the  above inequality have been done over the last three and a half decades. We refer to some of the works in this direction ~\cite{tertikas2003, tertikas2004, matskewich1997, opic1990}. To the best of our knowledge none of the work in literature is concerned with the case ~$p=1$ in ~\eqref{boundary hardy inequality}. The aim of this article is to establish appropriate inequalities for the case ~$p=1$.  In this article, our objective is to derive a boundary Hardy-type inequality within the space ~$BV(\Omega)$, where ~$\Omega$ is a bounded Lipschitz domain. Our approach initially involves establishing a fractional boundary Hardy inequality for the case ~$p=1$ and ~$s \geq  \frac{1}{2}$. Later, we utilise the well-known result of Brezis, Bourgain, and Mironescu as presented in ~\cite{bourgain2001} and ~\cite{bourgain2002} to obtain the Hardy inequality for functions of bounded variation.  We will recall the exact version of their results that will be required for us in the next section. 

Define the functions:
\begin{equation*}
    \mathcal{L}_{1}(t) = \frac{1}{1-\ln t}, \hspace{3mm} \forall \ t \in (0,1),
\end{equation*}
and recursively
\begin{equation*}
    \mathcal{L}_{m}(t) = \mathcal{L}_{1}(\mathcal{L}_{m-1}(t)), \hspace{3mm} \forall \ m \geq 2.
\end{equation*}

  Let us define the space of functions of bounded variation (also see, ~\cite{gariepy2015measure, giusti}):
\begin{defn}
    Let ~$\Omega \subset \mathbb{R}^d$ be an open set. A function ~$u \in L^{1}(\Omega)$ has bounded variation in ~$\Omega$ if
\begin{equation*}
    [u]_{BV(\Omega)} := \sup \left\{ \int_{\Omega} u(x)  div(\phi(x)) dx : \phi \in C^{1}_{c}(\Omega; \mathbb{R}^d), \ |\phi(x)| \leq 1 \ \text{on}  \ \Omega \right\} < \infty.
\end{equation*}
\end{defn}
We denote ~$BV(\Omega)$ the space of functions of bounded variation in ~$\Omega$ with the  norm \\ ~$||.||_{BV(\Omega)}$ on ~$BV(\Omega)$ as
$||u||_{BV(\Omega)} := [u]_{BV(\Omega)} + ||u||_{L^{1}(\Omega)}.$ Throughout this article ~$(u)_{\Omega}$ will denote the average of ~$u$ over ~$\Omega$ which is given by $(u)_{\Omega} = \frac{1}{|\Omega|} \int_{\Omega} u(y)dy.$\smallskip 

The following theorem is the main result of this article.
\begin{theorem}\label{th: main result}
    Let ~$\Omega$ be a bounded Lipschitz domain such that ~$\delta_{\Omega}(x) < R $ for all ~$x \in \Omega$, for some ~$R >0$ and ~$m \geq 2$ be a positive integer. Then there exists a constant ~$C=C(d, \Omega)>0$ such that  for all ~$u \in BV(\Omega)$, 
    \begin{multline}\label{main in}
         \int_{\Omega} \frac{|u(x) - (u)_{\Omega}|}{\delta_{\Omega}(x)}   \mathcal{L}_{1} \left( \frac{\delta_{\Omega}(x)}{R} \right) \cdots \mathcal{L}_{m-1} \left( \frac{\delta_{\Omega}(x)}{R} \right)  \mathcal{L}^{2}_{m} \left( \frac{\delta_{\Omega}(x)}{R} \right) dx  \leq C 2^{m} [u]_{BV(\Omega)}.
    \end{multline}
    Furthermore, for any $0<\alpha< \frac{1}{2}$, there exists a constant ~$C=C(d,\Omega)$ such that  for all ~$u \in BV(\Omega)$,
    \begin{multline}\label{series expansion main theorem}
     \sum_{m=2}^{\infty}   \alpha^{m}\int_{\Omega} \frac{|u(x) - (u)_{\Omega}|}{\delta_{\Omega}(x)}  \mathcal{L}_{1} \left( \frac{\delta_{\Omega}(x)}{R} \right) \cdots \mathcal{L}_{m-1} \left( \frac{\delta_{\Omega}(x)}{R} \right)  \mathcal{L}^{2}_{m} \left( \frac{\delta_{\Omega}(x)}{R} \right)dx \\  \leq C \left( \frac{4 \alpha^{2}}{1-2 \alpha} \right) [u]_{BV(\Omega)}. 
    \end{multline}
   The above inequality fails when ~$\alpha \geq 1$. 
\end{theorem}
 Following identity (see, ~\cite[Section 2, equation (2.1)]{tertikas2003})
\begin{equation}\label{derivative of L fn}
    \frac{d}{dt} \mathcal{L}_{m}(t) = \frac{1}{t} \mathcal{L}_{1}(t) \cdots \mathcal{L}_{m-1}(t) \mathcal{L}^{2}_{m}(t), \hspace{3mm} \text{where}  \ m \geq 2
\end{equation}
plays an important role in establishing the previous theorem among other key ingredients.
The constant appearing in the previous result may not be sharp. Since, the constant functions are in ~$BV(\Omega)$, this justifies the presence of ~$(u)_{\Omega}$ on the left hand side of ~\eqref{main in}. Also, ~$W^{1,1}(\Omega) \subset BV(\Omega)$ (see, ~\cite[Chapter 5]{gariepy2015measure}) and 
$\int_{\Omega} |\nabla u(x)| dx = [u]_{BV(\Omega)}, \hspace{3mm} \forall \ u \in W^{1,1}(\Omega).$
Therefore, Theorem ~\ref{th: main result} hold true for any ~$u \in W^{1,1}(\Omega)$.

\smallskip
Let ~$m \geq 1$, ~$\beta>1$ and $R$ be as in the previous theorem, then there exists a constant ~$C=C(\beta)>0$ (see, ~\eqref{corollary appen}, Appendix \ref{appen}) such that
\begin{equation}\label{beta condition}
  \mathcal{L}^{\beta}_{m}\left( \frac{\delta_{\Omega}(x)}{R} \right) \leq C   \mathcal{L}_{m} \left( \frac{\delta_{\Omega}(x)}{R} \right)   \mathcal{L}^{2}_{m+1} \left( \frac{\delta_{\Omega}(x)}{R} \right), \hspace{3mm} \forall \ x \in \Omega.
\end{equation}
Therefore, using Theorem ~\ref{th: main result} with ~$m+1 \geq 2 $ and the above inequality, we have the following immediate corollary:
\begin{cor}\label{cor 1}
     Let ~$\Omega$ be a bounded Lipschitz domain such that ~$\delta_{\Omega}(x) < R $ for all ~$x \in \Omega$, for some ~$R >0, ~ m \geq 1$ be a positive integer and ~$\beta>1$. Then there exists a constant ~$C=C(d, \Omega, \beta)>0$ such that  for all ~$u \in BV(\Omega)$,
     \begin{multline}
         \bigintsss_{\Omega} \frac{|u(x) - (u)_{\Omega}|}{\delta_{\Omega}(x)} \mathcal{L}_{1} \left( \frac{\delta_{\Omega}(x)}{R} \right) \cdots \mathcal{L}_{m-1} \left( \frac{\delta_{\Omega}(x)}{R} \right)  \mathcal{L}^{\beta}_{m} \left( \frac{\delta_{\Omega}(x)}{R} \right)    dx  \leq C 2^{m}[u]_{BV(\Omega)},
     \end{multline}
     with the convention that for ~$m=1$, \ ~$\mathcal{L}_{1} \left( \frac{\delta_{\Omega}(x)}{R} \right) \cdots \mathcal{L}_{m-1} \left( \frac{\delta_{\Omega}(x)}{R} \right)  \mathcal{L}^{\beta}_{m} \left( \frac{\delta_{\Omega}(x)}{R} \right) = \mathcal{L}_{1}^\beta\left( \frac{\delta_{\Omega}(x)}{R} \right).$
      Furthermore, for any ~$0<\alpha< \frac{1}{2}$, there exists a constant ~$C=C(d,\Omega, \beta)$ such that  for all ~$u \in BV(\Omega)$,
    \begin{multline}
     \sum_{m=2}^{\infty}  \alpha^{m} \int_{\Omega} \frac{|u(x) - (u)_{\Omega}|}{\delta_{\Omega}(x)}  \mathcal{L}_{1} \left( \frac{\delta_{\Omega}(x)}{R} \right) \cdots \mathcal{L}_{m-1} \left( \frac{\delta_{\Omega}(x)}{R} \right)  \mathcal{L}^{\beta}_{m} \left( \frac{\delta_{\Omega}(x)}{R} \right) dx \\  \leq C \left( \frac{4 \alpha^{2}}{1-2 \alpha} \right) [u]_{BV(\Omega)}.
    \end{multline}
   The above inequality fails whenever  ~$\alpha \geq 1$  or $0<\beta \leq 1$. 
\end{cor}

The above corollary does not hold true when ~$\beta=1$, can be illustrated by considering a non constant function ~$u \in BV(\Omega)$ such that ~$u$ takes non zero constant near the boundary of ~$\Omega$, resulting in the left-hand side of the inequality in the corollary becoming infinite. This failure illustrates the optimality of the aforementioned corollary with regards to the choice of ~$\beta$. Furthermore, it can be easily verified that the constant in the corollary, denoted as ~$C= C(d, \Omega, \beta) \to \infty$ as ~$\beta \to 1$. This is because ~$C = C(\beta)$ defined in ~\eqref{beta condition} tends to ~$\infty$ as ~$\beta \to 1$ (see, ~\eqref{cor 5.1 1} with ~$\theta = \beta -1$, Appendix ~\ref{appen}). 

\smallskip

Let ~$\rho: (0,1) \to \mathbb{R}$ be a measurable function satisfying ~$\rho > 0, ~ \rho \left( t \right) \to 0$ as ~$ t \to 0$, $\beta >1$ and for some constant ~$C>0$,
\begin{equation}\label{rho condition}
    \mathcal{L}^{1 + \rho \left(  t \right)}_{m}\left( t \right) \leq  C  \mathcal{L}_{m} \left( t \right)   \mathcal{L}^{\beta}_{m+1} \left( t  \right), \hspace{3mm} \forall \ t \in (0,1).
\end{equation}
Using the definition of ~$\mathcal{L}_{m}$, for any ~$m \geq 1$ and define ~$\mathcal{L}_{0}(t) := t$, we obtain
\begin{equation*}
    \left(1-\ln \left( \mathcal{L}_{m}(t) \right) \right)^{\beta} \leq C \left(1- \ln \left( \mathcal{L}_{m-1}(t) \right) \right)^{\rho(t)}.
\end{equation*}
Then, taking $\ln$ both sides, we obtain
\begin{equation*}
 \rho^{*}(t) :=  \frac{\beta \ln \left( 1-\ln \left( \mathcal{L}_{m}(t) \right) \right) }{ \ln  \left(1- \ln \left( \mathcal{L}_{m-1}(t) \right) \right) } \leq \rho (t) + \frac{\ln C}{ \ln \left(1- \ln \left( \mathcal{L}_{m-1}(t) \right) \right) }.
\end{equation*}
Since, ~$\frac{\ln C}{ \ln \left( 1- \ln \left( \mathcal{L}_{m-1}(t) \right) \right) } \to 0$ as ~$t \to 0$. Therefore, it is not necessary to consider this term or we can assume ~$C=1$. We also observe that ~$\rho^{*} \left( t \right) \to 0$ as ~$t \to 0$ and 
\begin{equation*}
     \mathcal{L}^{1 + \rho^{*} \left(  t \right)}_{m}\left( t \right) = \mathcal{L}_{m} \left( t \right)   \mathcal{L}^{\beta}_{m+1} \left( t  \right), \hspace{3mm} \forall \ t \in (0,1).
\end{equation*}
This implies that the function ~$\rho^{*}$ is optimal in the inequality ~\eqref{rho condition} with the choice of ~$\rho$. Therefore, we again present the following corollary which is a consequence of previous Corollary:
\begin{cor}\label{cor 2}
     Let ~$\Omega$ be a bounded Lipschitz domain such that ~$\delta_{\Omega}(x) < R $ for all ~$x \in \Omega$, for some ~$R >0, ~ m \geq 1$ be positive integers and ~$\beta>1$. Then there exists a constant ~$C=C(d, \Omega , \beta)>0$ such that for all  ~$u \in BV(\Omega)$,
     \begin{multline}
         \bigintsss_{\Omega} \frac{|u(x) - (u)_{\Omega}|}{\delta_{\Omega}(x)} \mathcal{L}_{1} \left( \frac{\delta_{\Omega}(x)}{R} \right) \cdots \mathcal{L}_{m-1} \left( \frac{\delta_{\Omega}(x)}{R} \right)  \mathcal{L}^{1+ \rho^{*} \left( \frac{\delta_{\Omega}(x)}{R} \right)}_{m} \left( \frac{\delta_{\Omega}(x)}{R} \right)dx \\ \leq C 2^{m}[u]_{BV(\Omega)}.
     \end{multline}
      Furthermore, the above inequality fails when ~$0<\beta \leq 1$ in the definition of ~$\rho^{*}$.
\end{cor}

We can illustrate the failure of the above corollary for ~$0<\beta \leq 1$ in the definition of ~$\rho^{*}$ by selecting a non-zero function ~$u \in BV(\Omega)$ that remains constant near the boundary ~$\partial \Omega$. This choice makes the left-hand side of the inequality in the above corollary to become infinite while the righ hand side is finite.

 \smallskip

We also establish a similar type of Hardy inequality in fractional Sobolev space when ~$p=1$. The next theorem that can be treated as an independent result in its own rights,  and serves as a crucial component in establishing Theorem ~\ref{th: main result}. In particular, we prove the following theorem: let ~$s \in (0,1)$ and ~$p \geq 1$, we define the Gagliardo fractional seminorm 
\begin{equation*}
    [u]_{W^{s,p}(\Omega)} := \left( \int_{\Omega} \int_{\Omega}  \frac{|u(x)-u(y)|^{p}}{|x-y|^{d+sp}}dxdy \right)^{\frac{1}{p}} .
\end{equation*}

\begin{theorem}
    \label{th: intermediate th} Let ~$\Omega$ be a bounded Lipschitz domain such that ~$\delta_{\Omega}(x) < R $ for all ~$x \in \Omega$, for some ~$R >0, ~ m \geq 2$ be positive integers and ~$\frac{1}{2} \leq s <1$. Then there exists a constant ~$C=C(d, \Omega)>0$ such that
\begin{multline}
    \int_{\Omega} \frac{|u(x)|}{\delta_\Omega^s(x)} \mathcal{L}_{1} \left( \frac{\delta_{\Omega}(x)}{R} \right) \cdots \mathcal{L}_{m-1} \left( \frac{\delta_{\Omega}(x)}{R} \right)  \mathcal{L}^{2}_{m} \left( \frac{\delta_{\Omega}(x)}{R} \right)dx \\ \leq C2^{m}(1-s)[u]_{W^{s,1}(\Omega)}  +  C 2^{m}||u||_{L^{1}(\Omega)},  \hspace{3mm}  \forall \ u \in W^{s,1}(\Omega).
\end{multline}
Furthermore, for any ~$0<  \alpha < \frac{1}{2}$, there exists a constant ~$C=C(d, \Omega)$ such that
\begin{multline}\label{inter}
  \sum_{m=2}^{\infty}  \alpha^{m} \int_{\Omega} \frac{|u(x)|}{\delta_\Omega^s(x)}  \mathcal{L}_{1} \left( \frac{\delta_{\Omega}(x)}{R} \right) \cdots \mathcal{L}_{m-1} \left( \frac{\delta_{\Omega}(x)}{R} \right)  \mathcal{L}^{2}_{m} \left( \frac{\delta_{\Omega}(x)}{R} \right)dx  \\ \leq C \left( \frac{4 \alpha^{2}}{1-2 \alpha} \right) \left\{ (1-s)[u]_{W^{s,1}(\Omega)}  +   ||u||_{L^{1}(\Omega)} \right\} , \hspace{3mm}  \forall \ u \in W^{s,1}(\Omega) .
\end{multline}
Also, the above inequality fails when ~$ \alpha \geq 1$. 
\end{theorem}

The inequalities \eqref{beta condition} and \eqref{rho condition} can also be applied in the Theorem ~\ref{th: intermediate th} to obtain a similar type of corollaries which is obtained for Theorem ~\ref{th: main result} in Corollary ~\ref{cor 1} and Corollary ~\ref{cor 2}.

\smallskip

The result that  comes closer to our present work is that of Barbatis, Filippas and Tertikas in ~\cite{tertikas2003} where they obtained a series expansion for ~$L^{p}$ Hardy inequalities in ~$\mathbb{R}^d, ~ p>1$ involving distance function from the boundary of domain ~$\Omega \subset \mathbb{R}^d$.
For more literature on Hardy type inequalities we refer to ~\cite{adi2002, adi2009, AdiPurbPro2023, adimurthi2024fractional, adiTrudinger2024, tertikas2003, debdip2020, Brezis1997, Brezis19971, sandeep2008, lu2024, lu2022, lu20223, lulam, Nguyen2019} and to the works mentioned there in. \smallskip

The article is organized in the following way: In Section ~\ref{preliminaries}, we present preliminary lemmas and notations that will be utilized to prove Theorem ~\ref{th: main result} and Theorem ~\ref{th: intermediate th}. In section ~\ref{main results proof d=1}, we prove Theorem ~\ref{th: main result} and Theorem ~\ref{th: intermediate th} in dimension one.  Section ~\ref{proof of main result} contains the proof of the main theorem, Theorem ~\ref{th: main result} in dimension $d \geq 2$  which follows from Theorem~\ref{th: intermediate th} and utilizes results from Brezis, Bourgain, and Mironescu (see, Lemma~\ref{bv function}). In section ~\ref{failure}, the counterexample that shows \eqref{series expansion main theorem} in Theorem ~\ref{th: main result} and \eqref{inter} in Theorem \ref{th: intermediate th} fails for $\alpha \geq 1$ is provided.


\section{Notations and Preliminaries}\label{preliminaries} 
In this section, we introduce the notations and preliminary lemmas that will be used in proving Theorem ~\ref{th: main result} and Theorem ~\ref{th: intermediate th}.  All the lemmas proved in this section are essentially known results in literature. Throughout this article, we shall use the following notations: 
\begin{itemize}
\item ~$\mathbb{R}^{d} $ will denote the Euclidean space of dimension ~$d$.

\item ~$s$ will always be understood to be in ~$(0,1)$.
    \item we denote ~$|\Omega|$ the Lebesgue measure of ~$\Omega \subset \mathbb{R}^{d} $.
      \item for any ~$f, ~ g: \Omega ( \subset \mathbb{R}^d) \to \mathbb{R}$, we denote ~$f \sim g$ if there exists ~$C_{1}, ~ C_{2}>0$ such that ~$C_{1}g(x) \leq f(x) \leq C_{2} g(x)$ for all ~$x \in \Omega$.
    \item ~$C>0$ will denote a generic constant that may change from line to line. 
\end{itemize}

Let ~$\Omega \subset \mathbb{R}^d$ be an open set and ~$s \in (0,1)$. For any ~$p \in [1, \infty)$, define the fractional Sobolev space
\begin{equation*}
W^{s,p}(\Omega) := \Big\{    u \in L^{p}(\Omega) : \int_{\Omega} \int_{\Omega}  \frac{|u(x)-u(y)|^{p}}{|x-y|^{d+sp}}  dxdy < \infty \Big\},
\end{equation*}
endowed with the norm
\begin{equation*}
    ||u||_{W^{s,p}(\Omega)} := \left(  [u]^{p}_{W^{s,p}(\Omega)} + ||u||^{p}_{L^{p}(\Omega)} \right)^{\frac{1}{p}}.
\end{equation*}
Let ~$W^{s,1}_{0}(\Omega)$ denotes the completion of ~$C^{\infty}_{c}(\Omega)$ with respect to the norm ~$||.||_{W^{s,1}(\Omega)}$.

\smallskip

{\bf A bounded domain with Lipschitz boundary}: Let ~$\Omega$ be a bounded Lipschitz domain. Then for each ~$x \in \partial \Omega$ there exists ~$r'_{x}>0$, an isometry ~$T_{x}$ of ~$\mathbb{R}^d$ and a Lipschitz function ~$\phi_{x} : \mathbb{R}^{d-1} \to \mathbb{R}$ such that
\begin{equation*}
    T_{x}(\Omega) \cap B_{r'_{x}}(T_{x}(x)) = \{ \xi : \xi_{d} > \phi_{x}(\xi') \} \cap B_{r'_{x}}(T_{x}(x))  . 
\end{equation*}

\smallskip

The next lemma proves fractional Poincar\'e inequality with some specific constant (see, ~\cite[page no. 80 (``fact")]{bourgain2002}) for any cube of side length ~$\lambda>0$. A more general version of this lemma is also available in ~\cite[Corollary 1]{mazya2002}. This lemma is helpful in proving Lemma ~\ref{flat case 1 d=1} and Lemma ~\ref{flat case 1}.

\begin{lemma}\label{poincare}
Let ~$d \geq 1, ~ \frac{1}{2} \leq s <1$ and ~$\Omega_{\lambda}$ be any cube of side length ~$\lambda>0$ in ~$\mathbb{R}^{d}$. Then there exists a constant ~$C_{d, Poin}=C_{d, Poin}(d)>0 $ such that
\begin{equation}
\fint_{\Omega_{\lambda}} |u(x)- (u)_{\Omega_{\lambda}}| dx \leq C_{d, Poin} \lambda^{s-d} (1-s)  \int_{\Omega_{\lambda}} \int_{\Omega_{\lambda}} \frac{|u(x)-u(y)|}{|x-y|^{d+s}} dxdy, \hspace{3mm} \forall \ u \in W^{s,1}(\Omega_{\lambda}).
\end{equation}
\end{lemma}

\begin{proof}
Let ~$\Omega$ be any unit cube. Then from ~\cite[page no. 80 (``fact")]{bourgain2002}, we have
\begin{equation*}
    \fint_{\Omega} |u(x)- (u)_{\Omega}| dx \leq C_{d, Poin} (1-s)  \int_{\Omega} \int_{\Omega} \frac{|u(x)-u(y)|}{|x-y|^{d+s}} dxdy,
\end{equation*}
where ~$C_{d, Poin}$ is the best fractional  Poincar\'e constant. Let us apply the above inequality to ~$u(\lambda x)$ instead of ~$u(x)$. This gives
    \begin{equation*}
         \fint_{\Omega}  \Big|u(\lambda x)-\fint_{\Omega} u(\lambda x) dx \Big|  dx   \leq C_{d, Poin}(1-s) \int_{\Omega} \int_{\Omega} \frac{|u(\lambda x) - u(\lambda y)|}{|x-y|^{d+s}} dxdy . 
    \end{equation*}
    Using the fact
    \begin{equation*}
        \fint_{\Omega} u(\lambda x) dx = \fint_{\Omega_{\lambda}} u(x) dx,
    \end{equation*}
    we have
    \begin{equation*}
         \fint_{\Omega} |u(\lambda x)-(u)_{\Omega_{\lambda}}| \ dx   \leq C_{d, Poin}(1-s)  \int_{\Omega} \int_{\Omega} \frac{|u(\lambda x) - u(\lambda y)|}{|x-y|^{d+s}} dxdy   . 
    \end{equation*}
    By changing the variable ~$X=\lambda x$ and ~$Y= \lambda y$, we obtain 
    \begin{equation*}
        \fint_{\Omega_{\lambda}} |u(x)-(u)_{\Omega_{\lambda}}| dx   \leq C_{d, Poin}\lambda^{s-d} (1-s) \int_{\Omega_{\lambda}} \int_{\Omega_{\lambda}} \frac{|u(x)-u(y)|}{|x-y|^{d+s}} dxdy.
    \end{equation*}
    This finishes the proof of the lemma.
\end{proof}

\smallskip

The next  lemma is an well known result in ~\cite{bourgain2001}. It plays a very crucial role in establishing Theorem ~\ref{th: main result}.

\begin{lemma}\label{bv function}
Let ~$u \in L^{1}(\Omega) \cap BV(\Omega)$. Then there exists  positive constants ~$C_{1}, C_{2}>0$ depends only on ~$\Omega$ such that
\begin{multline}
   C_{1}[u]_{BV(\Omega)} \leq \liminf_{\epsilon \to 1^+} ~\epsilon  \int_{\Omega} \int_{\Omega}  \frac{|u(x)-u(y)|}{|x-y|^{d+1- \epsilon}} dxdy \\ \leq \limsup_{\epsilon \to 1^+} ~\epsilon \int_{\Omega} \int_{\Omega} \frac{|u(x)-u(y)|}{|x-y|^{d+1- \epsilon}} dxdy \leq C_{2} [u]_{BV(\Omega)}. 
\end{multline}
\end{lemma}
\begin{proof}
    See \cite[corollary 5]{bourgain2001} for the proof.
\end{proof}
\begin{rem}\label{davilla remark}
    \normalfont D\'{a}vila in ~\cite{davila2002} established a stronger version of the above lemma. Infact,  D\'{a}vila in ~\cite[Theorem 1]{davila2002} proved that for a bounded Lipschitz domain ~$\Omega \subset \mathbb{R}^d$ and ~$u \in BV(\Omega)$, we have
    \begin{equation*}
        \lim_{s \to 1} (1-s) [u]_{W^{s,1}(\Omega)} = C_{BV,d} [u]_{BV(\Omega)},
    \end{equation*}
    where 
    \begin{equation*}
        C_{BV,d} = \int_{\mathbb{S}^{d-1}} |e.w| ds(w).
    \end{equation*}
    Here, ~$e \in \mathbb{R}^d$ denotes the unit vector and ~$ds$ is the surface measure on ~$\mathbb{S}^{d-1}$. But Lemma ~\ref{bv function} is sufficient to establish Theorem ~\ref{th: main result}.
\end{rem}

\smallskip

The following lemma establishing a connection to the average of ~$u$ over two disjoint sets. This technical step is very crucial in the proof of Lemma ~\ref{flat case 1 d=1} and Lemma ~\ref{flat case 1}.

\begin{lemma}\label{avg}
    Let ~$E$ and ~$F$ be measurable disjoint bounded set in ~$\mathbb{R}^d$ and $G$ be a cube of side length ~$\lambda>0$ such that ~$E \cup F \subset G$. Then 
    \begin{equation}
        |(u)_{E} - (u)_{F}| \leq C_{d, Poin} \lambda^{s-d} (1-s) \left( \frac{|G|}{\min \{ |E|, |F| \} }  \right) \int_{G} \int_{G} \frac{|u(x)-u(y)|}{|x-y|^{d+s}} dxdy,
    \end{equation}
    where ~$C_{d, Poin}$ is best fractional Poincar\'e constant for unit cube (the same ~$C_{d, Poin}$ as in Lemma \ref{poincare}).
\end{lemma}
\begin{proof}
    Let us consider:
    \begin{multline*}
        |(u)_{E} - (u)_{F}| \leq |(u)_{E} - (u)_{G}|+ |(u)_{F} - (u)_{G}| \\  \leq  \fint_{E} | u(x) - (u)_{G} |   dx   +   \fint_{F} | u(x) - (u)_{G} |  dx  \leq \frac{1}{\min \{ |E|, |F| \} } \int_{E \cup F} | u(x) - (u)_{G} |  dx  . 
    \end{multline*}
    In the second inequality above we have used triangle inequality for the integrals. 
    Since, ~$E \cup F \subset G$ we have
    \begin{equation*}
        |(u)_{E} - (u)_{F}| \leq  \frac{1}{\min \{ |E|, |F| \} } \int_{G} | u(x) - (u)_{G} |  dx.
    \end{equation*}
    Using Lemma ~\ref{poincare}, we have
    \begin{multline*}
        |(u)_{E} - (u)_{F}| \leq \left( \frac{|G|}{\min \{ |E|, |F| \} }  \right) \fint_{G} | u(x) - (u)_{G} |  dx \\ \leq C_{d, Poin} \lambda^{s-d} (1-s) \left( \frac{|G|}{\min \{ |E|, |F| \} }  \right) \int_{G} \int_{G} \frac{|u(x)-u(y)|}{|x-y|^{d+s}} dxdy.
    \end{multline*}
    This finishes the proof of lemma.
\end{proof}

\smallskip

The next lemma establishes approximation result involving the functions of bounded variation ~$BV(\Omega)$, where ~$\Omega$ is a bounded Lipschitz domain. This lemma is useful in establishing Theorem ~\ref{th: main result}. Note that it doesn't show that ~$W^{s,1}(\Omega) \cap BV(\Omega)$ is dense in ~$BV(\Omega)$(as full norm convergence is not true).

\begin{lemma}\label{density}
    Let ~$\Omega \subset \mathbb{R}^d$ be a bounded Lipschitz domain and ~$u \in BV(\Omega)$. Then there exists a sequence of functions ~$\{ u_{n} \}_{n=1}^{\infty} \subset W^{s,1}(\Omega) \cap BV(\Omega)\cap C^\infty(\Omega)$ such that ~$u_{n} \to u$ in ~$L^{1}(\Omega)$ and 
\begin{equation}
 [u_{n}]_{BV(\Omega)} \to [u]_{BV(\Omega)}
     \hspace{3mm}  \text{as} \ n \to \infty. 
\end{equation}
\end{lemma}
\begin{proof}
Let ~$u \in BV(\Omega)$, then from ~\cite[Theorem 5.3]{gariepy2015measure}, there exists a sequence of functions ~$\{ u_{n} \}_{n=1}^{\infty} \subset C^{\infty}(\Omega) \cap BV(\Omega)$ such that ~$u_{n} \to u$ in $L^{1}(\Omega)$ and 
\begin{equation*}
    [u_{n}]_{BV(\Omega)} \to [u]_{BV(\Omega)} \hspace{3mm} \text{as} \ n \to \infty. 
\end{equation*}
It is sufficient to prove ~$ u_{n}  \in W^{s,1}(\Omega)$, for all ~$n \in \mathbb{N}$. Since  ~$u_{n} \in C^{\infty}(\Omega) \cap BV(\Omega)$, using ~\cite[Chapter 1 (Example 1.2)]{giusti}, we have
\begin{equation*}
 \int_{\Omega} | \nabla u_{n} (x) | dx =   [u_{n}]_{BV(\Omega)} < \infty , \hspace{3mm} \forall \ n \in \mathbb{N},
\end{equation*}
which implies that ~$u_{n} \in W^{1,1}(\Omega)$ (see, ~\cite[Definition 4.2]{gariepy2015measure} for the definition of ~$W^{1,1}(\Omega)$). Using ~$W^{1,1}(\Omega) \subset W^{s,1}(\Omega)$ (see, ~\cite[Proposition 2.2]{di2012hitchhikers}), we have ~$u_{n} \in W^{s,1}(\Omega)$. Therefore, ~$u_{n} \in W^{s,1}(\Omega) \cap BV(\Omega)$ for all ~$n \in \mathbb{N}$.   
\end{proof}

\smallskip

The following lemma establishes a Poincar\'e type inequality on functions of bounded variation ~$BV(\Omega)$, where ~$\Omega$ is a bounded Lipschitz domain. This lemma is useful in establishing our main result.

\begin{lemma}\label{poincare in BV}
    Let ~$\Omega \subset \mathbb{R}^d$ be a bounded Lipschitz domain. Then there exists a constant ~$C_{BV, Poin}= C_{BV, Poin}(\Omega)>0$ such that
    \begin{equation}
        \int_{\Omega} |u(x) - (u)_{\Omega}| dx \leq C_{BV, Poin}[u]_{BV(\Omega)}, \hspace{3mm} \forall \ u \in BV(\Omega).
    \end{equation}
\end{lemma}
\begin{proof} 
See ~\cite[Theorem 3.2]{bergounioux2011} for the proof. If ~$\Omega$ is assumed to be a cube (which is an actual requirement for this article), then the lemma easily follows by choosing ~$p=1, s\rightarrow 1^{-}$ and as a combination of  Lemma ~\ref{poincare} and Remark ~\ref{davilla remark}. Further it can be easily derived that 
\begin{equation}\label{relation between constants}
   C_{BV, Poin} =   C_{d, Poin}C_{BV,d}. 
\end{equation}
\end{proof}

\smallskip

The next lemma establishes an inequality when any function ~$u \in W^{s,p}(\Omega)$ is multiplied by a test function. This lemma plays a crucial role in establishing Theorem ~\ref{th: intermediate th}. We denote by ~$C^{0, 1}(\Omega)$ the class of Lipschitz function ~$u : \Omega \to \mathbb{R}$ (see, ~\cite[Chapter 3, Definition 3.1]{gariepy2015measure}). 
\begin{lemma}\label{xi estimate}
    Let ~$\Omega$ be an open set in ~$\mathbb{R}^d$. Let us consider ~$u \in W^{s,p}(\Omega)$ and ~$\xi \in C^{0,1}(\Omega), ~ 0 \leq \xi \leq 1$. Then ~$\xi u \in W^{s,p}(\Omega)$ and for some constant ~$C=C(d,p,s,\Omega)>0$, 
    \begin{equation}
        || \xi u||_{W^{s,p}(\Omega)} \leq C ||u||_{W^{s,p}(\Omega)}.
    \end{equation}
\end{lemma}
\begin{proof}
    See ~\cite[Lemma 5.3]{di2012hitchhikers} for the proof. 
\end{proof}

\section{Proofs of the main results in dimension one}\label{main results proof d=1}
In this section, we present the proof of Theorem ~\ref{th: main result} and Theorem ~\ref{th: intermediate th} in dimension one. Establishing our main results in dimension one ~$(d=1)$ builds the foundation for extending the proof to higher dimensions and essentially all the major ideas can be explained with easy for this case. Extending them to higher dimensions is more technicality.  Also, we present quantitative estimate of the constants involved in this case. For simplicity, we first  establish the main results for the domain ~$\Omega= (0,2)$ and for any other general domain ($\Omega=(0,2D), D>0$) it can be obtained by translation and dilation of the domain ~$\Omega$.  \smallskip

The strategy is the following: The proof of Theorem \ref{th: main result} for ~$\Omega =(0,2D), ~ D>0$, is done in subsection ~\ref{3.3}, follows from the proof of Theorem ~\ref{th: intermediate th} for ~$\Omega =(0,2D), ~ D>0$  which is done in subsection ~\ref{3.2},.  The first part of the proof of Theorem ~\ref{th: intermediate th} for ~$\Omega =(0,2)$ will follow easily from Lemma ~\ref{flat case 1 d=1} which is done in subsection ~\ref{3.1}. Lemma ~\ref{estimate Aki and Li} and Lemma ~\ref{large n ineq} are basic inequalities that will be used to proof Lemma ~\ref{flat case 1 d=1}. 

\smallskip

For each ~$k \in \mathbb{Z}, ~ k \leq -1$ and ~$d=1$, define
\begin{equation*}
A_{k} := \{ x: 3^{k} \leq x < 3^{k+1} \}.
\end{equation*}
The next lemma establishes a basic inequality for each ~$A_{k}$. It gives a basic relation between for each ~$\mathcal{L}_{m}$ and ~$x \in A_{k}$. This lemma is helpful in proving Lemma ~\ref{flat case 1 d=1} and Lemma ~\ref{flat case 1}.

\begin{lemma}\label{estimate Aki and Li}
    For any ~$A_{k}, ~ R >1$ and ~$x \in A_{k}$, we have
    \begin{equation}
        \mathcal{L}_{1}\left( \frac{x}{R} \right) < \frac{1}{-k} =: \mathcal{Y}_{1}(k),
    \end{equation}
and for any ~$m \geq 2$,
\begin{equation}
    \mathcal{L}_{m} \left( \frac{x}{R} \right) < \frac{1}{1- \ln \left( \mathcal{Y}_{m-1}(k) \right)} =: \mathcal{Y}_{m}(k).
\end{equation}
\end{lemma}
\begin{proof}
Let ~$x \in A_{k}$. Then ~$x < 3^{k+1}$ which implies
\begin{equation*}
    \ln \left( \frac{x}{R} \right) = \ln \left( x \right) - \ln R < (k+1) \ln3 - \ln R < (k+1) \ln 3.
\end{equation*}
Therefore, we have
\begin{equation*}
    1- \ln \left( \frac{x}{R} \right) > 1-(k+1) \ln 3 > 1-(k+1)= -k.
\end{equation*}
From the definition of ~$\mathcal{L}_{1}$, we obtain
\begin{equation*}
    \mathcal{L}_{1} \left( \frac{x}{R} \right) = \frac{1}{1-\ln \left( \frac{x}{R} \right)} < \frac{1}{-k}= \mathcal{Y}_{1}(k).
\end{equation*}
Using the above inequality, we have
\begin{equation*}
    \ln \left( \mathcal{L}_{1} \left( \frac{x}{R} \right) \right) < \ln \left( \mathcal{Y}_{1}(k) \right).
\end{equation*}
So, from the definition of ~$\mathcal{L}_{2}(x)$, we have
\begin{equation*}
    \mathcal{L}_{2} \left( \frac{x}{R} \right) = \mathcal{L}_{1} \left( \mathcal{L}_{1}\left( \frac{x}{R} \right) \right) =  \frac{1}{1-\ln \left( \mathcal{L}_{1} \left( \frac{x}{R} \right) \right)} < \frac{1}{1- \ln \left( \mathcal{Y}_{1}(k) \right)} = \mathcal{Y}_{2}(k).
\end{equation*}
Therefore, recursively, we obtain for any ~$m \geq 2$,
\begin{equation*}
    \mathcal{L}_{m} \left( \frac{x}{R} \right) < \frac{1}{1- \ln \left( \mathcal{Y}_{m-1}(k) \right)} =\mathcal{Y}_{m}(k).
\end{equation*}
  This proves the lemma.
\end{proof}

The next lemma establishes a basic inequality for ~$k \leq -1$. This lemma is helpful in the proof of Lemma ~\ref{flat case 1 d=1} and Lemma ~\ref{flat case 1}.

\begin{lemma}\label{large n ineq}
For all ~$k \leq -1$, we have
     \begin{equation}
       \mathcal{Y}_{m}(k) - \mathcal{Y}_{m}(k-1) \geq \frac{\mathcal{Y}_{1}(k) \cdots \mathcal{Y}_{m-1}(k) \mathcal{Y}^{2}_{m}(k)}{2^{m+1}}  .
    \end{equation}
\end{lemma}
\begin{proof}
 Let ~$f: [0, 1] \to [0, \infty)$ be a differentiable function on ~$\left( 0,1 \right)$ such that
 \begin{equation*}
     f(x) =\mathcal{Y}_{m}(k-1+x), \hspace{3mm} \text{for some} \ k \leq -1.
 \end{equation*} 
 By mean value theorem, there exist ~$\gamma \in (0, 1)$ such that 
    \begin{equation*}
        f'(\gamma) = f(1)- f(0) =   \mathcal{Y}_{m}(k) - \mathcal{Y}_{m}(k-1).
    \end{equation*}
    Also,
    \begin{equation*}
        f'(\gamma) = \mathcal{Y}_{1}(k-1+ \gamma) \cdots \mathcal{Y}_{m-1}(k-1+ \gamma) \mathcal{Y}^{2}_{m}(k-1+ \gamma)
    \end{equation*}
follows easily from direct computations or using an induction argument.
Therefore, combining the above two inequalities, we have 
\begin{equation}\label{eqn Ym}
      \mathcal{Y}_{m}(k) - \mathcal{Y}_{m}(k-1) = \mathcal{Y}_{1}(k-1+ \gamma) \cdots \mathcal{Y}_{m-1}(k-1+ \gamma) \mathcal{Y}^{2}_{m}(k-1+ \gamma).
\end{equation} 
Since, ~$\left( 1 - \frac{1}{k}+ \frac{\gamma}{k} \right) \leq 2$ for all ~$k \leq -1$. Therefore, we have
\begin{equation*}
    \mathcal{Y}_{1}(k-1+ \gamma) = \frac{1}{-k+1-\gamma} = \frac{1}{(-k) \left(1- \frac{1}{k} + \frac{\gamma}{k} \right)} \geq  \frac{1}{2(-k)} = \frac{\mathcal{Y}_{1}(k)}{2} .
\end{equation*}
From the above inequality, we have ~$ \ln \left( \mathcal{Y}_{1}(k-1+ \gamma) \right) \geq \ln \left( \frac{\mathcal{Y}_{1}(k)}{2} \right) $. Using this in the definition of ~$\mathcal{Y}_{2}$, we obtain 
\begin{equation*}
    \mathcal{Y}_{2}(k-1+ \gamma) = \frac{1}{1- \ln \left( \mathcal{Y}_{1}(k-1+ \gamma) \right)} \geq   \frac{1}{1- \ln \left( \frac{\mathcal{Y}_{1}(k)}{2} \right)} .
\end{equation*}
But, we have
\begin{equation*}
    \frac{1}{1- \ln \left( \frac{\mathcal{Y}_{1}(k)}{2} \right)} = \frac{1}{1- \ln \left( \mathcal{Y}_{1}(k) \right) - \ln \left( \frac{1}{2} \right)}  > \frac{1}{2} \left( \frac{1}{1- \ln \left( \mathcal{Y}_{1}(k) \right) }  \right) = \frac{\mathcal{Y}_{2}(k)}{2}.
\end{equation*}
Here, we have used ~$2 \left( 1- \ln \left( \mathcal{Y}_{1}(k) \right)  \right) >  1- \ln \left( \mathcal{Y}_{1}(k) \right) - \ln \left( \frac{1}{2} \right)$. Therefore, combining the above two inequalities, we have
\begin{equation*}
    \mathcal{Y}_{2}(k-1+ \gamma) \geq \frac{\mathcal{Y}_{2}(k)}{2}. 
\end{equation*}
From the definition of ~$\mathcal{Y}_{m}$ and using recursively, we obtain
\begin{equation*}
    \mathcal{Y}_{m}(k-1+ \gamma) \geq \frac{\mathcal{Y}_{m}(k)}{2}, \hspace{3mm} \forall \ m \geq 2.
\end{equation*}
Hence, from ~\eqref{eqn Ym}, we have
\begin{equation*}
     \mathcal{Y}_{m}(k) - \mathcal{Y}_{m}(k-1) \geq \frac{\mathcal{Y}_{1}(k) \cdots \mathcal{Y}_{m-1}(k) \mathcal{Y}^{2}_{m}(k)}{2^{m+1}}  .
\end{equation*}
This proves the lemma.
\end{proof}

The next lemma is the main step towards the proof of Theorem ~\ref{th: intermediate th} for ~$\Omega =(0,2)$. Once the next lemma is established, Theorem ~\ref{th: intermediate th} for  ~$\Omega =(0,2)$ is fairly an easy consequence.

\begin{lemma}\label{flat case 1 d=1}
Let ~$R >1, ~ m  \geq 2$ be positive integers and ~$\frac{1}{2} \leq s < 1$. Then there exists a constant ~$C_{1, Poin}>0$ such that for all ~$u\in W^{s,1}((0,1))$
    \begin{multline}  
       \int_{0}^{1} \frac{|u(x)|}{x^{s}} \mathcal{L}_{1} \left( \frac{x}{R} \right) \cdots \mathcal{L}_{m-1} \left( \frac{x}{R} \right) \mathcal{L}^{2}_{m} \left( \frac{x}{R} \right)dx \\ \leq  C_{1, Poin}  ( 2^{3s+m} + 2^s)(1-s)[u]_{W^{s,1}((0,1))}+  2^{m+1} 3^{s}||u||_{L^{1}((0,1))}.
    \end{multline}
\end{lemma}
\begin{proof}
It is well known that ~$W^{s,1}((0,1)) = W^{s,1}_{0}((0,1))$ (see, ~\cite[Theorem 6.78]{leonibook}, as ~$sp <1$ with ~$p=1$ here). Therefore, it is sufficient to establish the following lemma for any ~$u \in C^{1}_{c}((0,1))$.
Let ~$u \in C^{1}_{c}((0,1))$ and fix any  ~$A_{k}$. Applying Lemma ~\ref{poincare} with  ~$ \frac{1}{2} \leq s<1, ~ \Omega = (\frac{1}{2} , \frac{3}{2})$, and ~$ \lambda = 2 \times 3^k$, we have
\begin{equation*}
    \fint_{A_{k}} |u(x)-(u)_{A_{k}}|  dx \leq C_{1, Poin}2^{s-1} 3^{k(s-1)} (1-s) [u]_{W^{s,1}(A_{k})},
\end{equation*}
where ~$C_{1, Poin}>0$ is as in Lemma ~\ref{poincare}. For ~$x \in A_{k}$ one has ~$ \frac{1}{x} \leq \frac{1}{3^{k}}$ which implies
\begin{multline*}
    \int_{A_{k}} \frac{|u(x)|}{x^{s}}  dx \leq\frac{1}{3^{k s}} \int_{A_{k}} |u(x)-(u)_{A_{k}} + (u)_{A_{k}}| dx \\
    \leq \frac{1}{3^{k s}} \int_{A_{k}} |u(x)-(u)_{A_{k}}|  dx + \frac{1}{3^{k s}} \int_{A_{k}} |(u)_{A)_{k}}|dx.
    \end{multline*}
  Now, using the previous inequality, we obtain
    \begin{equation*}
    \begin{split}
    \int_{A_{k}} \frac{|u(x)|}{x^{s}_{d}}  dx & \leq  \frac{|A_{k}|}{3^{k s}} \fint_{A_{k}}  |u(x)-(u)_{A_{k}}|  dx +  \frac{|A_{k}|}{3^{k s}} |(u)_{A_{k}}| \\
    & \leq C_{1, Poin}2^{s-1} \frac{2 \times 3^{k}}{3^{k s}} 3^{k(s-1)} (1-s) [u]_{W^{s,1}(A_{k})} +  2 \times 3^{k(1-s)} |(u)_{A_{k}}| \\ &
     \leq C_{1, Poin}2^s (1-s) [u]_{W^{s,1}(A_{k})} +  2 \times 3^{k(1-s)} |(u)_{A_{k}}| .
    \end{split}
\end{equation*}   
From Lemma ~\ref{estimate Aki and Li} and using ~$\mathcal{Y}_{1}(k) \cdots \mathcal{Y}_{m-1}(k) \mathcal{Y}^{2}_{m}(k) \leq 1$, we have
\begin{multline*}
    \int_{A_{k}} \frac{|u(x)|}{x^{s}} \mathcal{L}_{1} \left( \frac{x}{R} \right) \cdots \mathcal{L}_{m-1} \left( \frac{x}{R} \right) \mathcal{L}^{2}_{m} \left( \frac{x}{R} \right)   dx \leq C_{1, Poin}2^{s}(1-s) [u]_{W^{s,1}(A_{k})} \\ + 2\times3^{k(1-s)} \mathcal{Y}_{1}(k) \cdots \mathcal{Y}_{m-1}(k) \mathcal{Y}^{2}_{m}(k)  |(u)_{A_{k}}| .
\end{multline*}
Summing the above inequality from ~$k=\ell \in \mathbb{Z}^{-}$ to ~$-1$, we get
\begin{multline}\label{flat d=1 1}
\sum_{k= \ell}^{-1} \int_{A_{k}} \frac{|u(x)|}{x^{s}} \mathcal{L}_{1} \left( \frac{x}{R} \right) \cdots \mathcal{L}_{m-1} \left( \frac{x}{R} \right) \mathcal{L}^{2}_{m} \left( \frac{x}{R} \right)  dx    \leq  C_{1, Poin}2^s (1-s)  \sum_{k= \ell}^{-1} [u]_{W^{s,1}(A_{k})} \\ + 2\sum_{k= \ell}^{-1}  3^{k(1-s)} \mathcal{Y}_{1}(k) \cdots \mathcal{Y}_{m-1}(k) \mathcal{Y}^{2}_{m}(k) |(u)_{A_{k}}|  .
\end{multline} 
Independently, using triangle inequality, we have 
\begin{equation*}
    |(u)_{A_{k}}| \leq  |(u)_{A_{k+1}}| + |(u)_{A_{k}} - (u)_{A_{k+1}}|.
\end{equation*}
Since, ~$A_{k} \cup A_{k+1}$ is a interval of length ~$2 \times 3^{k+1}+ 2 \times 3^{k}$. Therefore, using Lemma ~\ref{avg} with ~$G= A_{k} \cup A_{k+1}$ and ~$\lambda= 8 \times 3^{k}$, we have  
\begin{multline*}
    |(u)_{A_{k}}| \leq  |(u)_{A_{k+1}}| + C_{1, Poin} (8^{s-1}3^{k(s-1)}) (1-s) [u]_{W^{s,1}(A_{k} \cup A_{k+1})}  \\ \leq |(u)_{A_{k+1}}| + C_{1, Poin} 2^{3s-3}3^{k(s-1)} (1-s) [u]_{W^{s,1}(A_{k} \cup A_{k+1})}   .\end{multline*}
Multiplying the above inequality by ~$3^{k(1-s)}$ and using the trivial estimate that ~$3^{1-s}>1$, we get
\begin{equation*}
   3^{k(1-s)} |(u)_{A_{k}}|  \leq   3^{(k+1)(1-s)}|(u)_{A_{k+1}}| + C_{1, Poin} 2^{3s-3} (1-s) [u]_{W^{s,1}(A_{k} \cup A_{k+1})}.
\end{equation*}
 Multiplying the above inequality with ~$\mathcal{Y}_{m}(k)$ and using ~$\mathcal{Y}_{m}(k) \leq 1$ for all ~$k \leq -1$, we obtain
\begin{equation*}
    3^{k(1-s)} \mathcal{Y}_{m}(k)   |(u)_{A_{k}}| \leq 3^{(k+1)(1-s)} \mathcal{Y}_{m}(k)  |(u)_{A_{k+1}}| + C_{1, Poin} 2^{3s-3}(1-s)[u]_{W^{s,1}(A_{k} \cup A_{k+1})}.
\end{equation*}
Summing the above inequality from ~$k=\ell \in \mathbb{Z}^{-}$ to ~$-2$, we get
\begin{multline*}
 \sum_{k=\ell}^{-2} 3^{k(1-s)} \mathcal{Y}_{m}(k)   |(u)_{A_{k}}| \leq \sum_{k=\ell}^{-2} 3^{(k+1)(1-s)} \mathcal{Y}_{m}(k)  |(u)_{A_{k+1}}| \\  + C_{1, Poin} 2^{3s-3}(1-s) \sum_{k=\ell}^{-2} [u]_{W^{s,1}(A_{k} \cup A_{k+1})} .
\end{multline*}
By changing sides, rearranging, and re-indexing, we get
\begin{multline*}
  3^{\ell(1-s)} \mathcal{Y}_{m}(\ell) |(u)_{A_{\ell}}|  +  \sum_{k=\ell+1}^{-2} 3^{k(1-s)} \left\{ \mathcal{Y}_{m}(k) - \mathcal{Y}_{m}(k-1) \right\} |(u)_{A_{k}}| \\ \leq 3^{(-1)(1-s)} \mathcal{Y}_{m}(-2) |(u)_{A_{-1}}| 
    +  C_{1, Poin} 2^{3s-3}(1-s) \sum_{k=\ell}^{-2} [u]_{W^{s,1}(A_{k} \cup A_{k+1})}.
\end{multline*}
Using the asymptotics (see, Lemma ~\ref{large n ineq})
\begin{equation*}
    \mathcal{Y}_{m}(k) - \mathcal{Y}_{m}(k-1) \geq \frac{\mathcal{Y}_{1}(k) \cdots \mathcal{Y}_{m-1}(k) \mathcal{Y}^{2}_{m}(k)}{2^{m+1}} ,
\end{equation*}
choose ~$-\ell$ large enough such that ~$|(u)_{A_{\ell}}|=0$ (as ~$u$ is assumed to be compactly supported), we obtain
\begin{multline*}
    \sum_{k=\ell}^{-2} \frac{3^{k(1-s)}}{2^{m+1}}  \mathcal{Y}_{1}(k) \cdots \mathcal{Y}_{m-1}(k) \mathcal{Y}^{2}_{m}(k) |(u)_{A_{k}}|\\ \leq \mathcal{Y}_{m}(-2) 3^{s-1} |(u)_{A_{-1}}|  +  C_{1, Poin} 2^{3s-3}(1-s) \sum_{k=\ell}^{-2} [u]_{W^{s,1}(A_{k} \cup A_{k+1})}.
\end{multline*}
 Adding  ~$\frac{3^{(-1)(1-s)}}{2^{m+1}}  \mathcal{Y}_{1}(-1) \cdots \mathcal{Y}_{m-1}(-1) \mathcal{Y}^{2}_{m}(-1) |(u)_{A_{-1}}| 
$
on both sides of the above inequality, we obtain
\begin{equation*}
\begin{split}
    \sum_{k=\ell}^{-1} \frac{3^{k(1-s)}}{2^{m+1}}  \mathcal{Y}_{1}(k) \cdots & \mathcal{Y}_{m-1}(k) \mathcal{Y}^{2}_{m}(k) |(u)_{A_{k}}| \\ &\leq  3^{s-1}\left\{\mathcal{Y}_{m}(-2) +\frac{1}{2^{m+1}}  \mathcal{Y}_{1}(-1) \cdots \mathcal{Y}_{m}(-1) \mathcal{Y}^{2}_{m}(-1)   \right\} |(u)_{A_{-1}}|   \\
    & \hspace{5mm} +  C_{1, Poin} 2^{3s-3}(1-s) \sum_{k=\ell}^{-2} [u]_{W^{s,1}(A_{k} \cup A_{k+1})} \\
    &\leq
    3^{s-1}|(u)_{A_{-1}}|   +  C_{1, Poin} 2^{3s-3}(1-s) \sum_{k=\ell}^{-2} [u]_{W^{s,1}(A_{k} \cup A_{k+1})}.
    \end{split}
\end{equation*}
In the last inequality we have used Lemma ~\ref{large n ineq} with ~$k=-1$.
Therefore, we have
\begin{multline}\label{flat d=1 lemma 2}
      \sum_{k=\ell}^{-1} 3^{k(1-s)} \mathcal{Y}_{1}(k) \cdots \mathcal{Y}_{m-1}(k) \mathcal{Y}^{2}_{m}(k) |(u)_{A_{k}}| \\\leq   2^{m+1} 3^{s-1}|(u)_{A_{-1}}|   +  C_{1, Poin} 2^{3s-3}2^{m+1}(1-s) \sum_{k=\ell}^{-2} [u]_{W^{s,1}(A_{k} \cup A_{k+1})}.
\end{multline}
Combining ~\eqref{flat d=1 1}, ~\eqref{flat d=1 lemma 2} together ~\eqref{se1} (see,  Appendix ~\ref{appen}) with ~$d=1$, yields
\begin{equation*}
\begin{split}
    \sum_{k=\ell}^{-1} \int_{A_{k}} \frac{|u(x)|}{x^{s}} \mathcal{L}_{1} & \left( \frac{x}{R} \right)  \cdots \mathcal{L}_{m-1} \left( \frac{x}{R} \right) \mathcal{L}^{2}_{m} \left( \frac{x}{R} \right)   dx  \leq   C_{1, Poin}2^s (1-s)  [u]_{W^{s,1}((0,1))}  \\ & \hspace{3mm} +
  2\left\{    2^{m+1} 3^{s-1}|(u)_{A_{-1}}|   +  C_{1, Poin} 2^{3s-3}2^{m+1}(1-s) \sum_{k=\ell}^{-2} [u]_{W^{s,1}(A_{k} \cup A_{k+1})}\right\} \\
  &\leq C_{1, Poin}  ( 2^{3s+m} + 2^s)(1-s)[u]_{W^{s,1}((0,1))}+  2^{m+2} 3^{s-1}|(u)_{A_{-1}}|.
  \end{split}
\end{equation*}
Using  ~$|(u)_{A_{-1}}| \leq (3/2) ||u||_{L^{1}((0,1))}$ 
we have
\begin{multline*}
    \int_{0}^{1} \frac{|u(x)|}{x^{s}} \mathcal{L}_{1} \left( \frac{x}{R} \right) \cdots \mathcal{L}_{m-1} \left( \frac{x}{R} \right) \mathcal{L}^{2}_{m} \left( \frac{x}{R} \right)dx \\ \leq  C_{1, Poin}  ( 2^{3s+m} + 2^s)(1-s)[u]_{W^{s,1}((0,1))}+  2^{m+1} 3^{s}||u||_{L^{1}((0,1))}.
\end{multline*}
This proves the lemma.
\end{proof}

\subsection{Proof Theorem ~\ref{th: intermediate th} for ~\texorpdfstring{~$\Omega =(0,2)$}{omega in dim 1}}\label{3.1}

 From Lemma ~\ref{flat case 1 d=1}, ~$m \geq 2, ~ R>1$ and ~$u \in W^{s,1}((0,2))$, we obtain
\begin{equation}\label{bv hardy d=1 - 1}
\begin{split}
      \int_{0}^{1} \frac{|u(x)|}{x^{s}} \mathcal{L}_{1} \left( \frac{x}{R} \right) \cdots \mathcal{L}_{m-1} \left( \frac{x}{R} \right)  \mathcal{L}^{2}_{m}\left( \frac{x}{R} \right)dx  \leq & C_{1, Poin}   \left( 2^{3s+m} + 2^s\right)(1-s)[u]_{W^{s,1}((0,1))}\nonumber\\ 
      &  \hspace{3mm} +  2^{m+1} 3^{s}||u||_{L^{1}((0,1))}.
\end{split}
\end{equation}
In the previous step we have used the fact that restriction of any ~$W^{s,1}((0,2))$ function on the the interval ~$(0,1)$ is again a ~$W^{s,1}((0,1))$ function.
Now since
\begin{equation*}
    \delta_{(0,2)}(x) = \begin{cases}
        x, & 0<x <1 \\
        2-x, & 1 \leq x <2,
    \end{cases}
\end{equation*}
 we have
\begin{equation*}
\begin{split}
     \int_{0}^{2} \frac{|u(x)|}{\delta^{s}_{(0,2)}(x)}  \mathcal{L}_{1} & \left( \frac{\delta_{(0,2)}(x)}{R} \right) \cdots \mathcal{L}_{m-1} \left( \frac{\delta_{(0,2)}(x)}{R} \right) \mathcal{L}^{2}_{m} \left( \frac{\delta_{(0,2)}(x)}{R} \right)dx \\
     &= \int_{0}^{1} \frac{|u(x)|}{x^{s}} \mathcal{L}_{1} \left( \frac{x}{R} \right) \cdots \mathcal{L}_{m-1} \left( \frac{x}{R} \right)  \mathcal{L}^{2}_{m}\left( \frac{x}{R}\right)dx \\ 
     & \hspace{5mm} + \int_{1}^{2} \frac{|u(x)|}{(2-x)^{s}} \mathcal{L}_{1} \left( \frac{2-x}{R} \right) \cdots \mathcal{L}_{m-1} \left( \frac{2-x}{R} \right)  \mathcal{L}^{2}_{m}\left( \frac{2-x}{R} \right)dx.
     \end{split}
\end{equation*}
Using change of variable ~$2-x=z$ in the last of integral of above equation and using ~\eqref{bv hardy d=1 - 1}, we obtain
\begin{multline*}
    \int_{0}^{2} \frac{|u(x)|}{\delta^{s}_{(0,2)}(x)} \mathcal{L}_{1} \left( \frac{\delta_{(0,2)}(x)}{R} \right) \cdots \mathcal{L}_{m-1} \left( \frac{\delta_{(0,2)}(x)}{R} \right) \mathcal{L}^{2}_{m} \left( \frac{\delta_{(0,2)}(x)}{R} \right)dx \\ \leq C_{1, Poin}\left( 2^{3s+m+1} + 2^{s+1} \right)(1-s)[u]_{W^{s,1}((0,2))}+  2^{m+2} 3^{s}||u||_{L^{1}((0,2))}.
\end{multline*}
This finishes the proof of the first part. Now let ~$\alpha<\frac{1}{2}$ and summing from ~$m=2$ to ~$\infty$, we have
\begin{equation*}
\begin{split} 
  \sum_{m=2}^{\infty}  & \alpha^{m}  \int_{0}^{2} \frac{|u(x)|}{\delta^{s}_{(0,2)}(x)}   \mathcal{L}_{1} \left( \frac{\delta_{(0,2)}(x)}{R} \right) \cdots \mathcal{L}_{m-1} \left( \frac{\delta_{(0,2)}(x)}{R} \right) \mathcal{L}^{2}_{m} \left( \frac{\delta_{(0,2)}(x)}{R} \right)  dx \\  & \leq \left\{ 2^{3s+1}\frac{4\alpha^2}{1-2\alpha} +2^{s+1}\frac{\alpha^2}{1-\alpha}\right\} C_{1, Poin}(1-s)   [u]_{W^{s,1}((0,2))} + 2^{2} \times 3^s\frac{4\alpha^2}{1-2\alpha} ||u||_{L^{1}((0,2))}  \\
  & =: \mathcal{A}(s,\alpha)C_{1, Poin}(1-s)   [u]_{W^{s,1}((0,2))} + \mathcal{B}(s,\alpha)||u||_{L^{1}((0,2))} . 
\end{split}
\end{equation*}
This proves the theorem for ~$\Omega = (0,2)$. 

\subsection{Proof Theorem ~\ref{th: intermediate th} for general domain in dimension one}\label{3.2}

Without any loss of generality assume ~$\Omega = (0, 2D)$ for ~$D >0$. Scaling appropriately the results in last subsection it is easy to that for ~$R>D$, clearly
\begin{multline*}
    \int_{0}^{2D} \frac{|u(x)|}{\delta^{s}_{(0,2D)}(x)} \mathcal{L}_{1} \left( \frac{\delta_{(0,2D)}(x)}{R} \right) \cdots \mathcal{L}_{m-1} \left( \frac{\delta_{(0,2D)}(x)}{R} \right) \mathcal{L}^{2}_{m} \left( \frac{\delta_{(0,2D)}(x)}{R} \right)dx \\ \leq  C_{1, Poin}\left( 2^{3s+m+1} + 2^{s+1}\right)(1-s)[u]_{W^{s,1}((0,2D))}+  \frac{2^{m+2} 3^{s}}{D^s}||u||_{L^{1}((0,2D))}.
\end{multline*}
The proof finishes of the first part. Let ~$\alpha<\frac{1}{2}$ and summing from ~$m=2$ to ~$\infty$, we have
\begin{equation*}
   \begin{split} 
  \sum_{m=2}^{\infty} \alpha^{m}  \int_{0}^{2D}  \frac{|u(x)|}{\delta^{s}_{(0,2D)}(x)} &   \mathcal{L}_{1} \left( \frac{\delta_{(0,2D)}(x)}{R} \right) \cdots \mathcal{L}_{m-1} \left( \frac{\delta_{(0,2D)}(x)}{R} \right) \mathcal{L}^{2}_{m} \left( \frac{\delta_{(0,2D)}(x)}{R} \right)  dx \\  & \leq \left\{ 2^{3s+1}\frac{4\alpha^2}{1-2\alpha} +2^{s+1}\frac{\alpha^2}{1-\alpha}\right\} C_{1, Poin}(1-s)   [u]_{W^{s,1}((0,2D))} \\& \hspace{5mm} +\left( 2^{2} \times 3^s \frac{4\alpha^2}{1-2\alpha} \right) \frac{1}{D^{s}} ||u||_{L^{1}((0,2D))}  \\
  & = \mathcal{A}(s,\alpha)C_{1, Poin}(1-s)   [u]_{W^{s,1}((0,2D))} + \frac{\mathcal{B}(s,\alpha)}{D^{s}} ||u||_{L^{1}((0,2D))} . 
\end{split} 
\end{equation*}

\subsection{Proof of Theorem ~\ref{th: main result} for general domain in dimension one}\label{3.3}  From Theorem ~\ref{th: intermediate th} with ~$d=1$, ~$\Omega= (0,2D), ~ D >0$ and ~$R > D$, we have
\begin{multline*}
    \int_{0}^{2D} \frac{|u(x)|}{\delta^{s}_{(0,2D)}(x)} \mathcal{L}_{1} \left( \frac{\delta_{(0,2D)}(x)}{R} \right) \cdots \mathcal{L}_{m-1} \left( \frac{\delta_{(0,2D)}(x)}{R} \right) \mathcal{L}^{2}_{m} \left( \frac{\delta_{(0,2D)}(x)}{R} \right)dx \\ \leq  C_{1, Poin}\left( 2^{3s+m+1} + 2^{s+1}\right)(1-s)[u]_{W^{s,1}((0,2D))}+  \frac{2^{m+2} 3^{s}}{D^s}||u||_{L^{1}((0,2D))}.
\end{multline*}
Using Fatou's lemma, we have
\begin{equation*}
\begin{split}
    \int_{0}^{2D} & \frac{|u(x)|}{\delta_{(0,2D)}(x)}  \mathcal{L}_{1} \left( \frac{\delta_{(0,2D)}(x)}{R} \right) \cdots \mathcal{L}_{m-1} \left( \frac{\delta_{(0,2D)}(x)}{R} \right) \mathcal{L}^{2}_{m} \left( \frac{\delta_{(0,2D)}(x)}{R} \right)dx \\ & \leq \liminf_{s \to 1}  \int_{0}^{2D} \frac{|u(x)|}{\delta^{s}_{(0,2D)}(x)} \mathcal{L}_{1} \left( \frac{\delta_{(0,2D)}(x)}{R} \right) \cdots \mathcal{L}_{m-1} \left( \frac{\delta_{(0,2D)}(x)}{R} \right) \mathcal{L}^{2}_{m} \left( \frac{\delta_{(0,2D)}(x)}{R} \right) dx \\ & \leq \liminf_{s \to 1} \left\{ C_{1, Poin}\left( 2^{3s+m+1} + 2^{s+1}\right)(1-s)[u]_{W^{s,1}((0,2D))}+  \frac{2^{m+2} 3^{s}}{D^s}||u||_{L^{1}((0,2D))} \right\}.
\end{split}
\end{equation*}
From Lemma ~\ref{bv function} and Remark ~\ref{davilla remark}, we have
\begin{multline*}
     \int_{0}^{2D} \frac{|u(x)|}{\delta_{(0,2D)}(x)} \mathcal{L}_{1} \left( \frac{\delta_{(0,2D)}(x)}{R} \right) \cdots \mathcal{L}_{m-1} \left( \frac{\delta_{(0,2D)}(x)}{R} \right) \mathcal{L}^{2}_{m} \left( \frac{\delta_{(0,2D)}(x)}{R} \right)     dx \\ \leq C_{1, Poin} C_{BV,1} \left( 2^{m+4} + 2^{2} \right)[u]_{BV((0,2D))}+ \frac{3 \times 2^{m+2}}{D}  ||u||_{L^{1}((0,2D))}.
\end{multline*}
Therefore, from Lemma ~\ref{poincare in BV} and using ~$[u-(u)_{(0,2D)}]_{BV((0,2D))} = [u]_{BV((0,2D))}$, we have
\begin{multline*}
     \int_{0}^{2D} \frac{|u(x) - (u)_{(0,2D)}|}{\delta_{(0,2D)}(x)} \mathcal{L}_{1} \left( \frac{\delta_{(0,2D)}(x)}{R} \right) \cdots \mathcal{L}_{m-1} \left( \frac{\delta_{(0,2D)}(x)}{R} \right) \mathcal{L}^{2}_{m} \left( \frac{\delta_{(0,2D)}(x)}{R} \right) dx \\ \leq   C_{1, Poin} C_{BV,1} \left( 2^{m+4} + 2^{2} + 3 \times 2^{m+3}   \right)[u]_{BV((0,2D))}.
\end{multline*}
Let ~$\alpha<\frac{1}{2}$ and summing from ~$m=2$ to ~$\infty$, we have
\begin{multline*}
  \sum_{m=2}^{\infty} \alpha^{m} \int_{0}^{2D} \frac{|u(x) - (u)_{(0,2D)}|}{\delta_{(0,2D)}(x)}  \mathcal{L}_{1} \left( \frac{\delta_{(0,2D)}(x)}{R} \right) \cdots \mathcal{L}_{m-1} \left( \frac{\delta_{(0,2D)}(x)}{R} \right) \mathcal{L}^{2}_{m} \left( \frac{\delta_{(0,2D)}(x)}{R} \right) dx \\ \leq C_{1, Poin} C_{BV,1} \left\{ 2^{5} \times 5 \frac{\alpha^{2}}{1-2 \alpha} + 2^{2}  \frac{\alpha^{2}}{1- \alpha}   \right\}  [u]_{BV((0,2D))}.
\end{multline*}
This proves the Theorem ~\ref{th: main result} in dimension $1$.

\section{Proof of the main results in dimension ~\texorpdfstring{~$d \geq 2$}{higher dimension}}\label{proof of main result}

In this section, we prove Theorem ~\ref{th: main result} and Theorem ~\ref{th: intermediate th} in dimension ~$d \geq 2$. Initially, we establish Theorem ~\ref{th: intermediate th} for the flat boundary case  which is done in Lemma ~\ref{flat case 1}. Consequently, in subsection ~\ref{4.1} we employ patching techniques to prove Theorem ~\ref{th: intermediate th}. Finally, we establish our main result, Theorem ~\ref{th: main result} in subsection ~\ref{4.2} using Theorem ~\ref{th: intermediate th} and Lemma ~\ref{bv function}. 

\smallskip

Let ~$\Omega_{n} = (-n,n)^{d-1} \times (0,1)$, where ~$n \in \mathbb{N}$. For each ~$k \in \mathbb{Z}$ and ~$k \leq -1$, set
\begin{equation*}
    A_{k} := \{ (x',x_{d}) : \  x' \in (-n,n)^{d-1}, \ 3^{k} \leq x_{d} < 3^{k+1} \} .
\end{equation*}
Then, we have ~$ \Omega_{n}= \bigcup_{k = - \infty}^{-1} A_{k}$. Again, we  further divide each ~$A_{k}$ into disjoint cubes each of side length ~$2 \times 3^k$ (say ~$A^{i}_{k}$). Then
\begin{equation*}
    A_{k} = \bigcup_{i=1}^{ 3^{(-k)(d-1)} n^{d-1}} A^{i}_{k}  .
\end{equation*}
For simplicity, let ~$\sigma_{k} =  3^{(-k)(d-1)} n^{d-1}$. Then
\begin{equation*}
    A_{k} = \bigcup_{i = 1}^{\sigma_{k}} A^{i}_{k}  .
\end{equation*}

The next lemma proves Theorem ~\ref{th: intermediate th} when the domain is ~$\mathbb{R}^{d}_{+}$ and the test functions supported on ~$ \Omega_{n} = (-n,n)^{d-1} \times (0,1)$, where ~$n \in \mathbb{N}$.

\begin{lemma}\label{flat case 1}
Let $ \Omega_{n} = (-n,n)^{d-1} \times (0,1)$ for some ~$n \in \mathbb{N}, ~R >1$, $m \geq 2$ be positive integers and $\frac{1}{2} \leq s < 1$. Then there exists a constant $C= C(d)>0$ such that for all ~$ u \in W^{s,1}(\Omega_{n})$,
    \begin{multline}  
       \int_{\Omega_{n}} \frac{|u(x)|}{x^{s}_{d}} \mathcal{L}_{1} \left( \frac{x_{d}}{R} \right) \cdots \mathcal{L}_{m-1} \left( \frac{x_{d}}{R} \right) \mathcal{L}^{2}_{m} \left( \frac{x_{d}}{R} \right)  dx \leq \\  C 2^{m} \left\{ (1-s)[u]_{W^{s,1}(\Omega_{n})}  +  ||u||_{L^{1}(\Omega_{n})} \right\}.
    \end{multline}
\end{lemma}
\begin{proof}
Since, $W^{s,1}(\Omega_{n}) = W^{s,1}_{0}(\Omega_{n})$ (see, ~\cite[Theorem 6.78]{leonibook}). Therefore, it is sufficient to establish the following lemma for any $u \in C^{1}_{c}(\Omega_{n})$. Let $u \in C^{1}_{c}(\Omega_{n})$ and fix any  ~$A^{i}_{k}$. Then ~$A^{i}_{k}$ is a translation of ~$(3^{k}, 3^{k+1})^{d}$. Applying Lemma ~\ref{poincare} with  ~$ \frac{1}{2} \leq s<1, ~ \Omega = \left( \frac{1}{2} , \frac{3}{2} \right)^d$, and ~$ \lambda = 2 \times 3^k$ and using translation invariance, we have
\begin{equation*}
    \fint_{A^{i}_{k}} |u(x)-(u)_{A^{i}_{k}}|  dx \leq C_{d, Poin} 2^{s-d} 3^{k(s-d)} (1-s) [u]_{W^{s,1}(A^{i}_{k})},
\end{equation*}
where $C_{d, Poin}>0$ is a constant as in Lemma ~\ref{poincare}. Let ~$x = (x',x_d) \in A^{i}_{k}$.  Then ~$x_d \geq 3^k$ which implies ~$ \frac{1}{x_{d}} \leq \frac{1}{3^{k}}$. Therefore, we have
\begin{multline*}
    \int_{A^{i}_{k}} \frac{|u(x)|}{x^{s}_{d}}  dx \leq\frac{1}{3^{k s}} \int_{A^{i}_{k}} |u(x)-(u)_{A^{i}_{k}} + (u)_{A^{i}_{k}}| dx \\
    \leq \frac{1}{3^{k s}} \int_{A^{i}_{k}} |u(x)-(u)_{A^{i}_{k}}|  dx + \frac{1}{3^{k s}} \int_{A^{i}_{k}} |(u)_{A^{i}_{k}}|dx.
    \end{multline*}
  Now, using the previous inequality, we obtain
    \begin{equation*}\label{eqn Aki}
    \begin{split}
    \int_{A^{i}_{k}} \frac{|u(x)|}{x^{s}_{d}}  dx  \leq  \frac{|A^{i}_{k}|}{3^{k s}}  
& \fint_{A^{i}_{k}} |u(x)-  (u)_{A^{i}_{k}}|   dx  +  \frac{|A^{i}_{k}|}{3^{k s}} |(u)_{A^{i}_{k}}| \\
    &\leq C_{d, Poin}2^{s-d} \frac{2^{d}3^{kd}}{3^{k s}} 3^{k(s-d)} (1-s) [u]_{W^{s,1}(A^{i}_{k})} +  2^{d} 3^{k(d-s)} |(u)_{A^{i}_{k}}|  
   \\ &= C_{d, Poin} 2^{s} (1-s) [u]_{W^{s,1}(A^{i}_{k})} + 2^{d} 3^{k(d-s)} |(u)_{A^{i}_{k}}|  .
   \end{split}
\end{equation*}   
From Lemma ~\ref{estimate Aki and Li} and using $\mathcal{Y}_{1}(k) \cdots \mathcal{Y}_{m-1}(k) \mathcal{Y}^{2}_{m}(k) \leq 1$, we have
\begin{multline*}
    \int_{A^{i}_{k}} \frac{|u(x)|}{x^{s}_{d}} \mathcal{L}_{1} \left( \frac{x_{d}}{R} \right) \cdots \mathcal{L}_{m-1} \left( \frac{x_{d}}{R} \right) \mathcal{L}^{2}_{m} \left( \frac{x_{d}}{R} \right)   dx \leq C_{d, Poin} 2^{s} (1-s) [u]_{W^{s,1}(A^{i}_{k})} \\ + 2^{d} 3^{k(d-s)} \mathcal{Y}_{1}(k) \cdots \mathcal{Y}_{m-1}(k) \mathcal{Y}^{2}_{m}(k)  |(u)_{A^{i}_{k}}| .
\end{multline*}
Summing the above inequality from ~$i=1$ to ~$\sigma_{k}$, we obtain
\begin{multline*}
    \int_{A_{k}} \frac{|u(x)|}{x^{s}_{d}} \mathcal{L}_{1} \left( \frac{x_{d}}{R} \right) \cdots \mathcal{L}_{m-1} \left( \frac{x_{d}}{R} \right) \mathcal{L}^{2}_{m} \left( \frac{x_{d}}{R} \right)   dx     \leq  C_{d, Poin} 2^{s} (1-s)  [u]_{W^{s,1}(A_{k})} \\ + 2^{d} 3^{k(d-s)} \mathcal{Y}_{1}(k) \cdots \mathcal{Y}_{m-1}(k) \mathcal{Y}^{2}_{m}(k)   \sum_{i=1}^{\sigma_{k}} |(u)_{A^{i}_{k}}|.
\end{multline*}
Summing the above inequality from ~$k= \ell \in \mathbb{Z}^{-}$ to ~$-1$, we get
\begin{multline}\label{eqnn1}
\sum_{k= \ell}^{-1} \int_{A_{k}} \frac{|u(x)|}{x^{s}_{d}} \mathcal{L}_{1} \left( \frac{x_{d}}{R} \right) \cdots \mathcal{L}_{m-1} \left( \frac{x_{d}}{R} \right) \mathcal{L}^{2}_{m} \left( \frac{x_{d}}{R} \right)  dx    \leq   C_{d, Poin} 2^{s} (1-s)  \sum_{k= \ell}^{-1} [u]_{W^{s,1}(A_{k})} \\ + 2^{d} \sum_{k= \ell}^{-1}  3^{k(d-s)} \mathcal{Y}_{1}(k) \cdots \mathcal{Y}_{m-1}(k) \mathcal{Y}^{2}_{m}(k) \sum_{i=1}^{\sigma_{k}} |(u)_{A^{i}_{k}}|  .
\end{multline} 
Let ~$A^{j}_{k+1}$ be a cube such that ~$A^{i}_{k}$ lies below the cube $A^{j}_{k+1}$. Independently, using triangle inequality, we have 
\begin{equation*}
    |(u)_{A^{i}_{k}}| \leq  |(u)_{A^{j}_{k+1}}| + |(u)_{A^{i}_{k}} - (u)_{A^{j}_{k+1}}|.
\end{equation*}
Choose a cube $G^{j}_{k+1}$ of side length $2 \times 3^{k+1}+ 2 \times 3^{k}$ such that $A^{i}_{k} \cup A^{j}_{k+1} \subset G^{j}_{k+1} $ and $G^{j}_{k+1} \subset A_{k} \cup A_{k+1}$. Therefore, using Lemma ~\ref{avg} with $E= A^{i}_{k}, ~ F = A^{j}_{k+1}$ and $G= G^{j}_{k+1}$ with $\lambda= 8 \times 3^{k} $, we have
\begin{equation*}
    |(u)_{A^{i}_{k}}| \leq  |(u)_{A^{j}_{k+1}}| + C_{d, Poin} 8^{s-d} 3^{k(s-d)} (1-s) [u]_{W^{s,1}(G^{j}_{k+1})}  .
\end{equation*}
Multiplying the above inequality by ~$3^{k(d-s)}$, we get
\begin{equation*}
   3^{k(d-s)} |(u)_{A^{i}_{k}}|  \leq 3^{k(d-s)}|(u)_{A^{j}_{k+1}}| + C_{d, Poin} 2^{3s-3d} (1-s)  [u]_{W^{s,1}(G^{j}_{k+1})}   .
\end{equation*}
Since, there are ~$3^{d-1}$ such ~$A^{i}_{k}$'s cubes lies below the cube ~$A^{j}_{k+1}$. Therefore, summing the above inequality from ~$i=3^{d-1}(j-1)+1$ to $3^{d-1}j$ and using $3^{d-1} \leq 3^{d-s}$, we obtain
\begin{multline*}
     3^{k(d-s)}  \sum_{i=3^{d-1}(j-1)+1}^{3^{d-1}j} |(u)_{A^{i}_{k}}| \leq 3^{d-1}3^{k(d-s)} |(u)_{A^{j}_{k+1}}| 
       + 3^{d-1} C_{d, Poin}  2^{3s-3d} (1-s)  [u]_{W^{s,1}(G^{j}_{k+1})}  \\ \leq 3^{(k+1)(d-s)}  |(u)_{A^{j}_{k+1}}| + C_{d, Poin} 2^{3s-3d} 3^{ d-1} (1-s) [u]_{W^{s,1}(G^{j}_{k+1})}   .
\end{multline*}
Again, summing the above inequality from ~$j=1$ to ~$\sigma_{k+1}$, using the fact that
\begin{equation*}
  \sum_{j=1}^{\sigma_{k+1}} \Bigg(  \sum_{i=3^{d-1}(j-1)+1}^{3^{d-1}j}  |(u)_{A^{i}_{k}}| \Bigg)  = \sum_{i=1}^{\sigma_{k}}  |(u)_{A^{i}_{k}}|,
\end{equation*}
and ~\eqref{se2} (See, Appendix ~\ref{appen}), we obtain
\begin{multline*}
3^{k(d-s)}  \sum_{i=1}^{\sigma_{k}} |(u)_{A^{i}_{k}}| \leq  3^{(k+1)(d-s)} \sum_{j=1}^{\sigma_{k+1}} |(u)_{A^{j}_{k+1}}| 
  + C_{d, Poin} 2^{3s-3d} 3^{ d-1} (1-s) \sum_{j=1}^{\sigma_{k+1}}  [u]_{W^{s,1}(G^{j}_{k+1})}    \\ \leq 3^{(k+1)(d-s)} \sum_{j=1}^{\sigma_{k+1}} |(u)_{A^{j}_{k+1}}| + C_{d, Poin} 2^{3s-3d+1} 3^{ d-1} (1-s)[u]_{W^{s,1}(A_{k} \cup A_{k+1})} .
\end{multline*}
Multiplying the above inequality with $\mathcal{Y}_{m}(k)$ and using $\mathcal{Y}_{m}(k) \leq 1$ for all $k \leq -1$, we obtain
\begin{multline*}
    3^{k(d-s)} \mathcal{Y}_{m}(k)  \sum_{i=1}^{\sigma_{k}} |(u)_{A^{i}_{k}}| \leq 3^{(k+1)(d-s)} \mathcal{Y}_{m}(k) \sum_{j=1}^{\sigma_{k+1}} |(u)_{A^{j}_{k+1}}| \\ +  C_{d, Poin} 2^{3s-3d+1} 3^{ d-1} (1-s)[u]_{W^{s,1}(A_{k} \cup A_{k+1})} .
\end{multline*}
For simplicity let ~$a_{k} = \sum_{i=1}^{\sigma_{k}} |(u)_{A^{i}_{k}}|$. Then, the above inequality will become
\begin{equation*}
  3^{k(d-s)} \mathcal{Y}_{m}(k)   a_{k} \leq 3^{(k+1)(d-s)} \mathcal{Y}_{m}(k)  a_{k+1} + C_{d, Poin} 2^{3s-3d+1} 3^{ d-1} (1-s) [u]_{W^{s,1}(A_{k} \cup A_{k+1})} .
\end{equation*}
Summing the above inequality from ~$k= \ell \in \mathbb{Z}^{-}$ to ~$-2$, we get
\begin{multline*}
 \sum_{k= \ell}^{-2} 3^{k(d-s)} \mathcal{Y}_{m}(k)   a_{k} \leq \sum_{k= \ell}^{-2} 3^{(k+1)(d-s)} \mathcal{Y}_{m}(k)  a_{k+1} \\ + C_{d, Poin} 2^{3s-3d+1} 3^{ d-1} (1-s) \sum_{k= \ell}^{-2} [u]_{W^{s,1}(A_{k} \cup A_{k+1})} .
\end{multline*}
By changing sides, rearranging, and re-indexing, we get
\begin{multline*}
  3^{l(d-s)} \mathcal{Y}_{m}(\ell) a_{\ell}  +  \sum_{k= \ell+1}^{-2} 3^{k(d-s)} \left\{ \mathcal{Y}_{m}(k) - \mathcal{Y}_{m}(k-1) \right\} a_{k} \\ \leq 3^{(-1)(d-s)} \mathcal{Y}_{m}(-2) a_{-1} 
    + C_{d, Poin} 2^{3s-3d+1} 3^{ d-1} (1-s) \sum_{k= \ell}^{-2} [u]_{W^{s,1}(A_{k} \cup A_{k+1})}.
\end{multline*}
Using the asymptotics (see, Lemma ~\ref{large n ineq})
\begin{equation*}
    \mathcal{Y}_{m}(k) - \mathcal{Y}_{m}(k-1) \geq \frac{\mathcal{Y}_{1}(k) \cdots \mathcal{Y}_{m-1}(k) \mathcal{Y}^{2}_{m}(k)}{2^{m+1}} ,
\end{equation*}
choose ~$-\ell$ large enough such that ~$|(u)_{A^{j}_{\ell}}|=0$ for all ~$j \in \{ 1, \cdots, \sigma_{\ell} \}$, we obtain
\begin{multline*}
    \sum_{k= \ell}^{-2} \frac{3^{k(d-s)}}{2^{m+1}}  \mathcal{Y}_{1}(k) \cdots  \mathcal{Y}_{m-1}(k)  
 \mathcal{Y}^{2}_{m}(k) a_{k} \\  \leq 3^{s-d} \mathcal{Y}_{m}(-2) a_{-1}    +  C_{d, Poin} 2^{3s-3d+1} 3^{ d-1} (1-s) \sum_{k= \ell}^{-2} [u]_{W^{s,1}(A_{k} \cup A_{k+1})}.
\end{multline*}
Adding ~$\frac{3^{(-1)(d-s)}}{2^{m+1}}  \mathcal{Y}_{1}(-1) \cdots \mathcal{Y}_{m-1}(-1) \mathcal{Y}^{2}_{m}(-1) a_{-1}$ on both sides of the above inequality, we obtain
\begin{equation*}
\begin{split}
   \sum_{k= \ell}^{-1} \frac{3^{k(d-s)}}{2^{m+1}}  \mathcal{Y}_{1}(k) \cdots \mathcal{Y}_{m-1}(k) & \mathcal{Y}^{2}_{m}(k) a_{k} \\ & \leq 3^{s-d} \left\{ \mathcal{Y}_{m}(-2) +  \frac{1}{2^{m+1}}  \mathcal{Y}_{1}(-1) \cdots \mathcal{Y}_{m-1}(-1) \mathcal{Y}^{2}_{m}(-1)  \right\}  a_{-1} \\ & \hspace{5mm}  +  C_{d, Poin} 2^{3s-3d+1} 3^{ d-1} (1-s) \sum_{k= \ell}^{-2} [u]_{W^{s,1}(A_{k} \cup A_{k+1})} \\ & \leq 3^{s-d} a_{-1} +  C_{d, Poin} 2^{3s-3d+1} 3^{ d-1} (1-s) \sum_{k= \ell}^{-2} [u]_{W^{s,1}(A_{k} \cup A_{k+1})} .
\end{split}
\end{equation*}
In the last inequality, we have used Lemma ~\ref{large n ineq} with $k=-1$. Therefore, we have
\begin{multline*}
    \sum_{k= \ell}^{-1} 3^{k(d-s)}  \mathcal{Y}_{1}(k) \cdots \mathcal{Y}_{m-1}(k) \mathcal{Y}^{2}_{m}(k) a_{k} \leq 2^{m+1} 3^{s-d} a_{-1} \\ + C_{d, Poin}  2^{3s-3d+m+2}  3^{ d-1} (1-s) \sum_{k= \ell}^{-2} [u]_{W^{s,1}(A_{k} \cup A_{k+1})} .
\end{multline*}
Putting the value of ~$a_{k}$ in the above inequality, we obtain
\begin{multline}\label{eqnn99}
    \sum_{k= \ell}^{-1} 3^{k(d-s)}  \mathcal{Y}_{1}(k) \cdots \mathcal{Y}_{m-1}(k) \mathcal{Y}^{2}_{m}(k) \sum_{i=1}^{\sigma_{k}} |(u)_{A^{i}_{k}}| \leq 2^{m+1} 3^{s-d}\sum_{j=1}^{\sigma_{-1}} |(u)_{A^{j}_{-1}}| \\ + C_{d, Poin}  2^{3s-3d+m+2}  3^{ d-1} (1-s) \sum_{k= \ell}^{-2} [u]_{W^{s,1}(A_{k} \cup A_{k+1})} .
\end{multline}
Combining ~\eqref{eqnn1} and ~\eqref{eqnn99} together ~\eqref{se1} (see,  Appendix ~\ref{appen}), yields
\begin{equation*}
\begin{split}
  \sum_{k= \ell}^{-1} & \int_{A_{k}}  \frac{|u(x)| }{x^{s}_{d}}  \mathcal{L}_{1} \left( \frac{x_{d}}{R} \right)  \cdots \mathcal{L}_{m-1} \left( \frac{x_{d}}{R} \right) \mathcal{L}^{2}_{m}  \left( \frac{x_{d}}{R} \right)   dx  \leq C_{d, Poin} 2^{s} (1-s)   [u]_{W^{s,1}(\Omega_{n})}  \\ & \hspace{3mm} +  2^{d} \left\{  2^{m+1} 3^{s-d} \sum_{j=1}^{\sigma_{-1}} |(u)_{A^{j}_{-1}}| + C_{d, Poin}  2^{3s-3d+m+2}  3^{ d-1} (1-s) \sum_{k= \ell}^{-2} [u]_{W^{s,1}(A_{k} \cup A_{k+1})} \right\} \\ & \leq C_{d, Poin} \left( 2^{s}+ 2^{3s-2d+m+3} 3^{d-1} \right)(1-s)  [u]_{W^{s,1}(\Omega_{n})} + 2^{m+d+1} 3^{s-d}\sum_{j=1}^{\sigma_{-1}} |(u)_{A^{j}_{-1}}|.
\end{split}
\end{equation*}
Also
\begin{equation*}
    \sum_{j=1}^{\sigma_{-1}} |(u)_{A^{j}_{-1}}| \leq \left( \frac{3}{2} \right)^{d} ||u||_{L^{1}(\Omega_{n})}.
\end{equation*}
Therefore, we have
\begin{multline*}
     \int_{\Omega_{n}} \frac{|u(x)|}{x^{s}_{d}} \mathcal{L}_{1} \left( \frac{x_{d}}{R} \right) \cdots \mathcal{L}_{m-1} \left( \frac{x_{d}}{R} \right) \mathcal{L}^{2}_{m} \left( \frac{x_{d}}{R} \right)  dx \\ \leq   C_{d, Poin} \left( 2^{s}+ 2^{3s-2d+m+3} 3^{d-1} \right)(1-s)[u]_{W^{s,1}(\Omega_{n})}  +  2^{m+1} 3^{s} ||u||_{L^{1}(\Omega_{n})}  .
\end{multline*}
This proves the lemma.
\end{proof}

\subsection{Proof of Theorem ~\ref{th: intermediate th}}\label{4.1}
Let ~$\Omega$ be a bounded Lipschitz domain. Consider the definition of bounded Lipschitz domain defined in Section ~\ref{preliminaries}. For simplicity, let ~$T_{x}$ be the identity map. Then 
\begin{equation*}
  \Omega \cap B_{r'_{x}}(x) = \{ \xi = (\xi',\xi_{d})  : \xi_{d} > \phi_{x}(\xi') \} \cap B_{r'_{x}}(x)  
\end{equation*}
and ~$\partial \Omega \subset \cup_{x \in \partial \Omega} B_{r'_{x}}(x)$. Choose ~$0<r_{x}<1$ such that ~$r_{x} \leq r'_{x}$ and for all ~$y \in \Omega \cap B_{r_{x}}(x)$, there exists ~$z \in \partial \Omega \cap B_{r_{x}}(x)$ satisfying ~$\delta_{\Omega}(y)= |y-z|$. Then ~$\partial \Omega \subset \cup_{x \in \partial \Omega} B_{r_{x}}(x)$. Since ~$\partial \Omega$ is compact, there exists ~$x_{1}, \cdots, x_{n} \in \partial \Omega$ such that
\begin{equation*}
    \partial \Omega  \subset \bigcup_{i=1}^{n} B_{r_{i}}(x_{i}) ,
\end{equation*}
where ~$r_{x_{i}}= r_{i}$.  

\smallskip

Let ~$u \in W^{s,1}(\Omega)$. Let ~$\Omega \subset \cup_{i=0}^{n} \Omega_{i}$ where ~$\overline{\Omega}_{0} \subset \Omega$ and ~$\Omega_{i} = B_{r_{i}}(x_{i})$ for all ~$ 1 \leq i \leq n $. Let ~$ \{ \eta_{i} \}_{i=0}^{n}$ be the associated partition of unity. Then
\begin{equation*}
    u = \sum_{i=0}^{n} u_{i} \hspace{.3cm} \text{where} \  u_{i} = \eta_{i} u.
\end{equation*}
From Lemma ~\ref{xi estimate}, we have
\begin{equation*}
    || u_{i} ||_{W^{s,1}(\Omega \cap \Omega_{i} )} \leq C ||u||_{W^{s,1}(\Omega \cap \Omega_{i})}, \hspace{3mm} \forall \ 0 \leq i \leq 1.
\end{equation*}
Therefore, it is sufficient to prove Theorem ~\ref{th: intermediate th} for all $u_{i}, ~ 0 \leq i \leq n$. Since, ~$supp(u_{0}) \subset \Omega_{0}$ and for all ~$x \in \Omega_{0}$, 
 \begin{equation*}
     C_{1, Poin} \leq \delta_{\Omega}(x) \leq C_{2} \hspace{.3cm} \text{for some} \  C_{1, Poin}, C_{2} >0 .
 \end{equation*}
 Therefore, using $\mathcal{L}_{m} \left( \frac{\delta_{\Omega}(x)}{R} \right) \leq 1$ for all $m \geq 1$, we have
 \begin{equation*}
      \int_{\Omega_{0}} \frac{|u(x)|}{\delta^{s}_{\Omega}(x)} \mathcal{L}_{1} \left( \frac{\delta_{\Omega}(x)}{R} \right) \cdots  \mathcal{L}_{m-1} \left( \frac{\delta_{\Omega}(x)}{R} \right) \mathcal{L}^{2}_{m} \left( \frac{\delta_{\Omega}(x)}{R} \right)  dx \leq C \int_{\Omega_{0}} |u(x)| dx .
 \end{equation*}
 For ~$1 \leq i \leq n, ~ supp(u_{i}) \subset \Omega \cap \Omega_{i} $. Consider the transformation ~$F: \mathbb{R}^d \to \mathbb{R}^d$ such that ~$F(x',x_{d}) = (x',x_{d} - \phi_{x_{i}}(x'))$ and ~$G = F^{-1}$  (see subsection ~\ref{lipschitzdomain}, Appendix ~\ref{appen}), then
  \begin{equation*}
      \delta_{\Omega}(x) \sim \xi_{d} \hspace{.3cm} \text{for all} \  x \in \Omega \cap \Omega_{i},
 \end{equation*}
where ~$F(x) = (\xi_{1}, \cdots, \xi_{d})$. Therefore, from Lemma ~\ref{flat case 1}, we have
\begin{equation*}
\begin{split}
     \int_{\Omega \cap \Omega_{i}} \frac{|u(x)|}{\delta^{s}_{\Omega}(x)} \mathcal{L}_{1} \left( \frac{\delta_{\Omega}(x)}{R} \right) & \cdots \mathcal{L}_{m-1} \left( \frac{\delta_{\Omega}(x)}{R} \right)  \mathcal{L}^{2}_{m} \left( \frac{\delta_{\Omega}(x)}{R} \right)   dx \\ &  \sim \bigintsss_{F(\Omega \cap \Omega_{i})} \frac{|u_{i} \circ G(\xi)|}{\xi^{s}_{d}} \mathcal{L}_{1} \left( \frac{\xi_{d}}{R} \right) \cdots \mathcal{L}_{m-1} \left( \frac{\xi_{d}}{R} \right) \mathcal{L}^{2}_{m} \left( \frac{\xi_{d}}{R} \right)  d \xi \\ & \leq C 2^{m}(1-s)  [u_{i} \circ G]_{W^{s,1}(F(\Omega \cap \Omega_{i}))}   +  C2^{m} ||u_{i} \circ G||_{L^{1}(F(\Omega \cap \Omega_{i}))}   \\&  = C2^{m}(1-s)   [u_{i}]_{W^{s,1}(\Omega \cap \Omega_{i})}   + C2^{m}  ||u_{i}||_{L^{1}(\Omega \cap \Omega_{i})}   .
\end{split}
\end{equation*}
Hence, combining all the above cases, we obtain the following inequality:
\begin{multline*}
       \int_{\Omega} \frac{|u(x)|}{\delta^{s}_{\Omega}(x)} \mathcal{L}_{1} \left( \frac{\delta_{\Omega}(x)}{R} \right) \cdots \mathcal{L}_{m-1} \left( \frac{\delta_{\Omega}(x)}{R} \right)  \mathcal{L}^{2}_{m} \left( \frac{\delta_{\Omega}(x)}{R} \right)  dx \\ \leq C 2^{m}(1-s)  [u]_{W^{s,1}(\Omega)} +  C2^{m}||u||_{L^{1}(\Omega)} .
\end{multline*}
Let $\alpha<\frac{1}{2}$ and summing from $m=2$ to $\infty$, we have
\begin{multline*}
  \sum_{m=2}^{\infty} \alpha^{m}  \int_{\Omega} \frac{|u(x)|}{\delta^{s}_{\Omega}(x)}  \mathcal{L}_{1} \left( \frac{\delta_{\Omega}(x)}{R} \right) \cdots \mathcal{L}_{m-1} \left( \frac{\delta_{\Omega}(x)}{R} \right)  \mathcal{L}^{2}_{m} \left( \frac{\delta_{\Omega}(x)}{R} \right)  dx \\ \leq C \left( \frac{4 \alpha^{2}}{1-2 \alpha} \right) (1-s)  [u]_{W^{s,1}(\Omega)} +  C||u||_{L^{1}(\Omega)} \sum_{m=2}^{\infty} ( 2 \alpha )^{m} \\ \leq C \left( \frac{4 \alpha^{2}}{1-2 \alpha} \right)(1-s)  [u]_{W^{s,1}(\Omega)}  +  C||u||_{L^{1}(\Omega)}   .
\end{multline*}
This finishes the proof of Theorem \ref{th: intermediate th}.

\subsection{Proof of Theorem ~\ref{th: main result}}\label{4.2} From Lemma ~\ref{density}, it is sufficient to prove Theorem ~\ref{th: main result} for $W^{s,1}(\Omega) \cap BV(\Omega)$. Let $u \in W^{s,1}(\Omega) \cap BV(\Omega)$.  From Theorem \ref{th: intermediate th}, we have
\begin{multline*}
       \int_{\Omega} \frac{|u(x)|}{\delta^{s}_{\Omega}(x)} \mathcal{L}_{1} \left( \frac{\delta_{\Omega}(x)}{R} \right) \cdots \mathcal{L}_{m-1} \left( \frac{\delta_{\Omega}(x)}{R} \right)  \mathcal{L}^{2}_{m} \left( \frac{\delta_{\Omega}(x)}{R} \right)  dx \\ \leq C 2^{m}(1-s)  [u]_{W^{s,1}(\Omega)} +  C2^{m}||u||_{L^{1}(\Omega)} .
\end{multline*}
Using Fatou's lemma, we have
\begin{multline*}
    \int_{\Omega} \frac{|u(x)|}{\delta_{\Omega}(x)} \mathcal{L}_{1} \left( \frac{\delta_{\Omega}(x)}{R} \right) \cdots \mathcal{L}_{m-1} \left( \frac{\delta_{\Omega}(x)}{R} \right)  \mathcal{L}^{2}_{m} \left( \frac{\delta_{\Omega}(x)}{R} \right)  dx \\ \leq \liminf_{s \to 1} \int_{\Omega} \frac{|u(x)|}{\delta^{s}_{\Omega}(x)} \mathcal{L}_{1} \left( \frac{\delta_{\Omega}(x)}{R} \right) \cdots  \mathcal{L}_{m-1} \left( \frac{\delta_{\Omega}(x)}{R} \right) \mathcal{L}^{2}_{m} \left( \frac{\delta_{\Omega}(x)}{R} \right)  dx \\ \leq C2^{m} \liminf_{s \to 1}(1-s)[u]_{W^{s,1}(\Omega)} +  C2^{m}||u||_{L^{1}(\Omega)}.
\end{multline*}
From Lemma ~\ref{bv function}, we have
\begin{equation*}
     \int_{\Omega} \frac{|u(x)|}{\delta_{\Omega}(x)} \mathcal{L}_{1} \left( \frac{\delta_{\Omega}(x)}{R} \right) \cdots \mathcal{L}_{m-1} \left( \frac{\delta_{\Omega}(x)}{R} \right)  \mathcal{L}^{2}_{m} \left( \frac{\delta_{\Omega}(x)}{R} \right)  dx \leq C 2^{m} \left( [u]_{BV(\Omega)} + ||u||_{L^{1}(\Omega)} \right).
\end{equation*}
Therefore, from Lemma ~\ref{poincare in BV} and using $[u-(u)_{\Omega}]_{BV(\Omega)} = [u]_{BV(\Omega)}$, we have
\begin{equation*}
     \int_{\Omega} \frac{|u(x)-(u)_{\Omega}|}{\delta_{\Omega}(x)} \mathcal{L}_{1} \left( \frac{\delta_{\Omega}(x)}{R} \right) \cdots \mathcal{L}_{m-1} \left( \frac{\delta_{\Omega}(x)}{R} \right)  \mathcal{L}^{2}_{m} \left( \frac{\delta_{\Omega}(x)}{R} \right)  dx \leq C 2^{m} [u]_{BV(\Omega)}  .
\end{equation*}
Let $\alpha<\frac{1}{2}$ and summing from $m=2$ to $\infty$, we have
\begin{multline*}
  \sum_{m=2}^{\infty} \alpha^{m} \int_{\Omega} \frac{|u(x)-(u)_{\Omega}|}{\delta_{\Omega}(x)}  \mathcal{L}_{1} \left( \frac{\delta_{\Omega}(x)}{R} \right) \cdots \mathcal{L}_{m-1} \left( \frac{\delta_{\Omega}(x)}{R} \right)  \mathcal{L}^{2}_{m} \left( \frac{\delta_{\Omega}(x)}{R} \right)  dx \\ \leq C \left( \frac{4 \alpha^{2}}{1-2 \alpha} \right)  [u]_{BV(\Omega)}. 
\end{multline*}
This proves our main result, Theorem ~\ref{th: main result}.

\smallskip

\section{Failure for \texorpdfstring{$\alpha  \geq 1$}{alpha geq 1} in Theorem ~\ref{th: main result} and Theorem ~\ref{th: intermediate th}} \label{failure}

In this section we prove the failure of Theorem ~\ref{th: main result} and ~\ref{th: intermediate th} for $\alpha \geq 1$. First we establish the failure in Theorem ~\ref{th: main result} and then we establish the failure in Theorem ~\ref{th: intermediate th} using Theorem ~\ref{th: main result}. To prove the failure of our main results when $\alpha \geq 1$, it is sufficient to establish for the domain $\Omega = (-2n,2n)^{d-1} \times (0,2)$, where $n \in \mathbb{N}$ and a function supported on $\Omega_{n} = (-n,n)^{d-1} \times (0,1)$. Let $u' \in C^{\infty}_{c}((-n,n)^{d-1}) $ and $u_{d}(x) = 1$ for all $x \in (0,2)$. For any $x = (x', x_{d}) \in \Omega$, define 
\begin{equation}\label{contradiction function}
    u(x) = u'(x') u_{d}(x_{d})= u'(x').
\end{equation}
Then $u \in BV(\Omega) \cap W^{s,1}(\Omega)$ and for any $x \in \Omega_{n}$, we have $\delta_{\Omega}(x) = x_{d}$. From ~\eqref{derivative of L fn}, we obtain
\begin{multline}
     \int_{\Omega_{n}} \frac{|u(x)|}{x_{d}} \mathcal{L}_{1} \left( \frac{x_{d}}{R} \right) \cdots \mathcal{L}_{m-1} \left( \frac{x_{d}}{R} \right) \mathcal{L}^{2}_{m} \left( \frac{x_{d}}{R} \right)  dx = \frac{1}{R} \int_{\Omega_{n}} |u(x)| \frac{d}{dx_{d}} \mathcal{L}_{m} \left( \frac{x_{d}}{R} \right) dx' dx_{d} \\ = \frac{1}{R} \int_{(-n,n)^{d-1} } |u'(x')| dx' \int_{0}^{1} \frac{d}{dx_{d}} \mathcal{L}_{m}(x_{d})  dx_{d} \\ =  \frac{1}{R} \left( \mathcal{L}_{m} \left( \frac{1}{R} \right) - \mathcal{L}_{m}(0) \right) \int_{(-n,n)^{d-1} } |u'(x')| dx' .
\end{multline}
From the definition of $\mathcal{L}_{m}$  and usisng ~\eqref{eqn Lm cont.} (see, Appendix ~\ref{appen}), we have $\mathcal{L}_{m} \left( \frac{1}{R} \right) \geq \frac{1}{(m+1)R}$  and $\mathcal{L}_{m}(0) = 0$. Therefore, from above inequality, we have
\begin{equation*}
    \int_{\Omega_{n}} \frac{|u(x)|}{x_{d}} \mathcal{L}_{1} \left( \frac{x_{d}}{R} \right) \cdots \mathcal{L}_{m-1} \left( \frac{x_{d}}{R} \right) \mathcal{L}^{2}_{m} \left( \frac{x_{d}}{R} \right)  dx \geq \frac{1}{(m+1)R^{2}}  \int_{(-n,n)^{d-1} } |u'(x')| dx' .
\end{equation*}
For any $\alpha \geq 1$, we have
\begin{multline}\label{failure 1 eq}
   \sum_{m=2}^{\infty} \alpha^{m} \int_{\Omega_{n}} \frac{|u(x)|}{x_{d}}  \mathcal{L}_{1} \left( \frac{x_{d}}{R} \right) \cdots \mathcal{L}_{m-1} \left( \frac{x_{d}}{R} \right) \mathcal{L}^{2}_{m} \left( \frac{x_{d}}{R} \right)  dx \\ \geq \frac{1}{R^{2}} \int_{(-n,n)^{d-1} } |u'(x')| dx'  \sum_{m=2}^{\infty} \frac{\alpha^{m}}{m+1}  = \infty.
\end{multline}
This proves that Theorem ~\ref{th: main result} fails when $\alpha \geq 1$. 

\smallskip

We will now establish that the Theorem ~\ref{th: intermediate th} fails when $\alpha \geq 1$. We will prove by using contradiction. Let $u \in W^{s,1}(\Omega) \cap BV(\Omega)$ be a function defined in ~\eqref{contradiction function}. Assume there exists a constant  $C= C(\Omega, d, \alpha)>0$ and $\alpha \geq 1$ such that
\begin{multline*}
    \sum_{m=2}^{\infty} \alpha^{m} \int_{\Omega} \frac{|u(x)|}{\delta^{s}_{\Omega}(x)}  \mathcal{L}_{1} \left( \frac{\delta_{\Omega}(x)}{R} \right) \cdots \mathcal{L}_{m-1} \left( \frac{\delta_{\Omega}(x)}{R} \right) \mathcal{L}^{2}_{m} \left( \frac{\delta_{\Omega}(x)}{R} \right)  dx \\ \leq C (1-s)[u]_{W^{s,1}(\Omega)} + C ||u||_{L^{1}(\Omega)}.
\end{multline*}
Using Fatou's lemma and Lemma ~\ref{bv function}, we have for all $m_{0} >2$,
\begin{equation*}
\begin{split}
    \sum_{m=2}^{m_{0}} \alpha^{m} & \int_{\Omega} \frac{|u(x)|}{\delta_{\Omega}(x)}  \mathcal{L}_{1} \left( \frac{\delta_{\Omega}(x)}{R} \right) \cdots \mathcal{L}_{m-1} \left( \frac{\delta_{\Omega}(x)}{R} \right) \mathcal{L}^{2}_{m} \left( \frac{\delta_{\Omega}(x)}{R} \right)  dx \\ &  \leq  \sum_{m=2}^{m_{0}}  \alpha^{m} \liminf_{s \to 1} \int_{\Omega} \frac{|u(x)|}{\delta^{s}_{\Omega}(x)}  \mathcal{L}_{1} \left( \frac{\delta_{\Omega}(x)}{R} \right) \cdots \mathcal{L}_{m-1} \left( \frac{\delta_{\Omega}(x)}{R} \right) \mathcal{L}^{2}_{m} \left( \frac{\delta_{\Omega}(x)}{R} \right)  dx \\ & \leq \liminf_{s \to 1} \sum_{m=2}^{\infty}  \alpha^{m}  \int_{\Omega} \frac{|u(x)|}{\delta^{s}_{\Omega}(x)}  \mathcal{L}_{1} \left( \frac{\delta_{\Omega}(x)}{R} \right) \cdots \mathcal{L}_{m-1} \left( \frac{\delta_{\Omega}(x)}{R} \right) \mathcal{L}^{2}_{m} \left( \frac{\delta_{\Omega}(x)}{R} \right)  dx \\ & \leq \liminf_{s \to 1} C (1-s)[u]_{W^{s,1}(\Omega)} + C ||u||_{L^{1}(\Omega)} \leq C [u]_{BV(\Omega)} + C ||u||_{L^{1}(\Omega)}.
\end{split}  
\end{equation*}
which is a contradiction (see, ~\eqref{failure 1 eq}). This proves that Theorem ~\ref{th: intermediate th} fails when $\alpha \geq 1$.

\section{Appendix}\label{appen}

\subsection{Domain above the graph of a Lipschitz function}\label{lipschitzdomain}
In this section, we will prove that if ~$\Omega$ is a domain above the graph of a Lipschitz function ~$\gamma: \mathbb{R}^{d-1} \to \mathbb{R}$ and ~$F : \mathbb{R}^d \to \mathbb{R}^d$ be a map given by ~$F(x)= (\xi', \xi_{d})$ where ~$\xi ' = x'$ and ~$ \xi_{d} = x_{d} - \gamma(x')$. Then
\begin{equation*}
    \delta_{\Omega}(x) \sim \xi_{d}
\end{equation*}
for all ~$x \in \Omega$, i.e., there exists ~$C_{1, Poin}, ~ C_{2} >0$ such that
\begin{equation*}
    C_{1, Poin} \xi_{d} \leq \delta_{D}(x) \leq C_{2} \xi_{d} . 
\end{equation*}
Let ~$\gamma: \mathbb{R}^{d-1} \to \mathbb{R} $ be a Lipschitz function and ~$M >0$ such that ~$x', ~ y' \in \mathbb{R}^{d-1}$, we have
\begin{equation*}
    |\gamma(x') - \gamma(y')| \leq M |x'-y' | .
\end{equation*}
Let ~$F: \mathbb{R}^d \to \mathbb{R}^d$ such that ~$F(x)= (F_{1}(x), \dots, F_{d}(x))= (x', x_{d}- \gamma(x'))$ where ~$x'=(x_{1}, \dots , x_{d-1})$. Then
\begin{multline*}
     |F(x)-F(y)|^{2} = |x'-y'|^{2} + |x_{n}-y_{n}-\gamma(x')+ \gamma(y')|^{2} \\
    \leq |x'-y'|^{2} + |x_{d}-y_{d}|^{2} + |\gamma(x') - \gamma(y')|^{2} + 2 <x_{d}-y_{d}, \gamma(y') - \gamma(x')> \\
    \leq |x-y|^{2} + M^2 |x'-y'|^{2} + |x_{d}-y_{d}|^{2} + |\gamma(x')-\gamma(y')|^{2}   \\
    \leq |x-y|^{2} + 2M^2 |x'-y'|^{2} + |x_{d}-y_{d}|^{2} 
    \leq (2M^{2}+2) |x-y|^{2} .
\end{multline*}
Let ~$C= (2M^{2}+2)^{1/2}$, then ~$|F(x)-F(y)| \leq C |x-y|$. Define ~$G(\xi) = F^{-1}(\xi) = (\xi',\xi_{d}+ \gamma(\xi'))$. Then ~$G$ is Lipschitz and ~$| G(\xi)-G(\eta)| \leq (2M^{2}+2)^{1/2} |\xi-\eta|$. Hence, there exist ~$C>0$ such that
\begin{equation*}
    \frac{1}{C} |x-y| \leq |F(x)-F(y)| \leq C |x-y| .
\end{equation*}
Let ~$\Omega = \{ x \in \mathbb{R}^d : x_{d} > \gamma(x') \}$ and ~$\partial \Omega = \{ x \in \mathbb{R}^d : x_{d} = \gamma(x') \}$. Then ~$F(\Omega) = \mathbb{R}^{d}_{+}$ and ~$F(\partial \Omega) = \partial \mathbb{R}^d_{+}$. Let ~$x \in \Omega$ and ~$y \in \partial \Omega$ such that
\begin{equation*}
    \delta_{\Omega}(x) = |x-y| = \text{inf} \{ |x- \eta| : \eta \in \partial \Omega \} .
\end{equation*}
Then ~$\delta_{\Omega}(x) = |x-y| \leq |x- \eta|$ for all ~$\eta \in \partial \Omega$. Therefore,
\begin{equation*}
    \delta_{\Omega}(x) = |x-y| \leq C |F(x)- F(\eta)| 
    \leq C |F(x)- \xi | 
\end{equation*}
for all ~$\xi \in \partial \mathbb{R}^d_{+}$. So, ~$\delta_{\Omega}(x) \leq C \underset{\xi \in \partial \mathbb{R}^{d}_{+}}{\inf}  \{ |F(x) - \xi| \} = C F_{d}(x)$. Let ~$F(x) = (\xi' , \xi_{d})$, Then, we have ~$\delta_{\Omega}(x) \leq C \xi_{d}$. Similarly, considering ~$G$ we get ~$C_{1, Poin} \xi_{d} \leq \delta_{\Omega}(x)$. Therefore, ~$C_{1, Poin} \xi_{d} \leq \delta_{\Omega}(x) \leq C \xi_{d} $.

\subsection{Some estimates}\label{some estimate} \textbf{(1)}  Let $\theta >0 $ and for any $m \geq 1$, we establish that there exists a constant $C=C(\theta)>0$ such that
\begin{equation}\label{corollary appen}
    \mathcal{L}^{\theta}_{m} (t) \leq C \mathcal{L}^{2}_{m+1}(t), \hspace{3mm} \forall \ t \in (0,1).
\end{equation}
First assume $0 < \theta \leq 1$ and let $\mathcal{L}_{m}(t) = e^{-x}$. Then if $t=0$ then $x \to \infty$ and if $t=1$ then $x = 0$. Define,
\begin{equation*}
    g_\theta(x) = \frac{e^{\theta x}}{(1+x)^{2}} = \left( \frac{1}{ 1- \ln \left( \mathcal{L}_{m}(t) \right)} \right)^{2} \frac{1}{\mathcal{L}^{\theta}_{m}(t)} =  \frac{\mathcal{L}^{2}_{m+1}(t)}{ \mathcal{L}^{\theta}_{m}(t)}, \hspace{3mm} \forall \ x \in (0, \infty) .
\end{equation*}
Clearly,
\begin{equation*}
    \frac{\mathcal{L}^{2}_{m+1}(t)}{ \mathcal{L}^{\theta}_{m}(t)} \geq \min_{x\in (0,\infty)} g_\theta(x) = g_\theta \left(-1+ \frac{2}{\theta} \right) = \left( \frac{\theta}{2} \right)^{2} e^{2- \theta}. 
\end{equation*}
Therefore, we have
\begin{equation}\label{cor 5.1 1}
    \mathcal{L}^{\theta}_{m} (t) \leq \left( \frac{2}{\theta} \right)^{2} e^{\theta- 2}  \mathcal{L}^{2}_{m+1}(t), \hspace{3mm} \forall \ t \in (0,1) .
\end{equation}
Now, assume $\theta > 1$. Then, $\theta = n_{1} + r$, where $r \in (0,1]$. Therefore, using $\mathcal{L}_{m}(t) \leq 1$ and above inequality, we have
\begin{equation*}
    \mathcal{L}^{\theta}_{m}(t) = \mathcal{L}^{n_{1}}_{m}(t) \mathcal{L}^{r}_{m}(t) \leq \mathcal{L}^{r}_{m}(t) \leq \left( \frac{2}{r} \right)^{2} e^{r- 2}  \mathcal{L}^{2}_{m+1}(t)
\end{equation*}
This establishes the required inequality for any $\theta >0$.

\smallskip

\textbf{(2)} Let $\mathcal{L}_{m}$ defined in the introduction section and $R>1$. We prove
\begin{equation}\label{eqn Lm cont.}
    \mathcal{L}_{m} \left(  \frac{1}{R} \right) \geq \frac{1}{(m+1)R}.
\end{equation}
Since, $1-\ln \left( \frac{1}{R} \right) = 1 + \ln R \leq 2R$. Therefore, we have
\begin{equation*}
    \mathcal{L}_{1} \left( \frac{1}{R} \right) = \frac{1}{1-\ln \left( \frac{1}{R} \right)} \geq  \frac{1}{2R}.
\end{equation*}
From above inequality, we have $\ln \left( \mathcal{L}_{1} \left( \frac{1}{R} \right) \right) \geq \ln \left( \frac{1}{2R}  \right)$. Therefore, we have
\begin{equation*}
    1- \ln \left( \mathcal{L}_{1} \left( \frac{1}{R} \right) \right) \leq 1-\ln \left( \frac{1}{2R}  \right) =  1+\ln \left( 2R  \right) \leq 3R.
\end{equation*}
Using the definition of $\mathcal{L}_{2}$, we have
\begin{equation*} 
    \mathcal{L}_{2} \left( \frac{1}{R} \right) =  \frac{1}{1- \ln \left( \mathcal{L}_{1} \left( \frac{1}{R} \right) \right) } \geq \frac{1}{3R}.
\end{equation*}
From the definition of $\mathcal{Y}_{m}$ and using recursively, we obtain
\begin{equation*}
     \mathcal{L}_{m} \left( \frac{1}{R} \right) \geq \frac{1}{(m+1)R}.
\end{equation*}
This establishes the inequality.

\smallskip

\textbf{(3)} Let $\Omega_{n} = (-n,n)^{d-1} \times (0,1)$ and $A_{k}$ as defined in Lemma ~\ref{flat case 1}. We aim to show that
\begin{equation}\label{se1}
     \sum_{k= \ell}^{-2} [u]_{W^{s,1}(A_{k} \cup A_{k+1})} \leq 2 [u]_{W^{s,1}(\Omega_{n})}.
\end{equation}
Consider two families of sets:
\begin{equation*}
  \mathcal{E}:=  \left\{ A_{k} \cup A_{k+1} : -k \ \text{is even and} \ k \leq -1  \right\}
\end{equation*}
and
\begin{equation*}
  \mathcal{O}:=  \left\{ A_{k} \cup A_{k+1} : -k \ \text{is odd and} \ k \leq -1  \right\}.
\end{equation*}
Then $\mathcal{E}$ and $\mathcal{O}$ are collection of mutually disjoint sets respectively. Define
\begin{equation*}
    \mathcal{F}_{e} := \bigcup_{\substack{k =\ell \\ -k \ \text{is even}}}^{-2} A_{k} \cup A_{k+1} \hspace{3mm} \text{and} \hspace{3mm} \mathcal{F}_{o} := \bigcup_{\substack{k =\ell \\ -k \ \text{is odd}}}^{-2} A_{k} \cup A_{k+1}.
\end{equation*} 
From the definition of $A_{k}$, we have $\mathcal{F}_{e} \subset \Omega_{n}$ and $\mathcal{F}_{o} \subset \Omega_{n}$. Therefore, we have
\begin{multline}\label{sum Ak A(k+1)}
    \sum_{k= \ell}^{-2} [u]_{W^{s,1}(A_{k} \cup A_{k+1})} = \sum_{\substack{k =\ell \\ -k \ \text{is even}}}^{-2} [u]_{W^{s,1}(A_{k} \cup A_{k+1})} + \sum_{\substack{k =\ell \\ -k \ \text{is odd}}}^{-2} [u]_{W^{s,1}(A_{k} \cup A_{k+1})} \\ \leq [u]_{W^{s,1}(\mathcal{F}_{e})} + [u]_{W^{s,1}(\mathcal{F}_{o})} \leq 2 [u]_{W^{s,1}(\Omega_{n})}.
\end{multline}
This establishes the desired inequality.

\smallskip

\textbf{(4)} Let $A^{i}_{k}$ and $A^{j}_{k+1}$ be the cubes defined in Lemma ~\ref{flat case 1} such that $A^{i}_{k}$ lies below the cube $A^{j}_{k+1}$. Let $G^{j}_{k+1}$ be a cube of side length $2 \times 3^{k+1}+ 2 \times 3^{k}$ such that $A^{i}_{k} \cup A^{j}_{k+1} \subset G^{j}_{k+1} $ and $G^{j}_{k+1} \subset A_{k} \cup A_{k+1}$. Also, there are $3^{d-1}$ cubes of side length $2 \times 3^{k}$ (like $A^{i}_{k}$) lies below the cube $A^{j}_{k+1}$ and the same cube $G^{j}_{k+1}$ will work for all $3^{d-1}$ such cubes (like $A^{i}_{k}$). Therefore, we will establish:
\begin{equation}\label{se2}
    \sum_{j=1}^{\sigma_{k+1}} [u]_{W^{s,1}(G^{j}_{k+1})} \leq 2 [u]_{W^{s,1}(A_{k} \cup A_{k+1})}.
\end{equation}
According to the construction of $G^{j}_{k+1}$, the families of sets $\{ G^{j}_{k+1} : \ j \ \text{is even and} \ 1 \leq j \leq \sigma_{k+1} \}$ and $\{ G^{j}_{k+1} : \ j  \ \text{is odd and} \ 1 \leq j \leq \sigma_{k+1} \}$ are collection of mutually disjoint sets respectively. Therefore, similarly as the previous case, we have
\begin{equation*}
 \sum_{j=1}^{\sigma_{k+1}} [u]_{W^{s,1}(G^{j}_{k+1})} = \sum_{\substack{j=1 \\ j \ \text{is even}}}^{\sigma_{k+1}} [u]_{W^{s,1}(G^{j}_{k+1})} + \sum_{\substack{j=1 \\ j \ \text{is odd}}}^{\sigma_{k+1}} [u]_{W^{s,1}(G^{j}_{k+1})} \leq    2 [u]_{W^{s,1}(A_{k} \cup A_{k+1})}.
\end{equation*}
This establishes our inequality.

\bigskip

\textbf{Acknowledgement:} We express our gratitude to the Department of Mathematics and Statistics at the Indian Institute of Technology Kanpur for providing conductive research environment. For this work, Adimurthi acknowledges support from IIT Kanpur, while P. Roy is supported by the Core Research Grant (CRG/2022/007867) of SERB. V. Sahu is grateful for the support received through MHRD, Government of India (GATE fellowship).

\end{document}